\documentclass[11pt,reqno]{amsart}

\usepackage{amsfonts}
\usepackage{amsmath,amscd}
\usepackage{amssymb}
\usepackage{tikz-cd}
\usepackage{amsthm}
\usepackage{cases}

\usepackage{galois}
\usepackage{chemarrow}
\usepackage[all]{xy}
\usepackage{graphicx}
\usepackage[colorlinks,
            linkcolor=blue,
            linktoc=all,
            anchorcolor=black,
            citecolor=red
            ]{hyperref}	
\usepackage{mathrsfs}
\usepackage{extarrows}
\usepackage{texnames}
\usepackage[nameinlink,capitalize]{cleveref}

\def\rank{{\textrm{rank}}}

\def\Aut{{\rm Aut}}

\def\End{{\rm End}}

\def\rank{{\rm rank}}

\def\Tr{{\rm Tr}}

\theoremstyle{plain}
\newtheorem{introtheorem}{Theorem}

\newtheorem*{introremark}{Remark}
\newtheorem{theorem}{Theorem}[section]

\newtheorem{proposition/example}[theorem]{Proposition/Example}
\newtheorem{proposition}[theorem]{Proposition}
\newtheorem{corollary}[theorem]{Corollary}
\newtheorem{lemma}[theorem]{Lemma}

\theoremstyle{definition}
\newtheorem{definition}[theorem]{Definition}
\newtheorem{remark}[theorem]{Remark}

\newtheorem{conjecture/question}[theorem]{Conjecture/Question}

\newtheorem{remark/definition}[theorem]{Remark/Definition}
\newtheorem{definition/notation}[theorem]{Definition/Notation}

\numberwithin{equation}{section}

\begin{document}
\title{\textbf{Simpson--Mochizuki Correspondence for $\lambda$-Flat Bundles}}

\author{Zhi Hu}

\address{ \textsc{School of Mathematics and Statistics, Nanjing University of Science and Technology, Nanjing 210094, P.R. China}
\endgraf \textsc{Department of Mathematics, Mainz University, 55128 Mainz, Germany}}

\email{halfask@mail.ustc.edu.cn; huz@uni-mainz.de}

\author{Pengfei Huang}

\address{\textsc{School of Mathematical Sciences, University of Science and Technology of China, Hefei 230026, P.R. China}\endgraf \textsc{Laboratoire J.A. Dieudonn\'e, Universit\'e C\^ote d'Azur, CNRS, 06108 Nice, France}\endgraf \textsc{Mathematisches Institut, Ruprecht-Karls Universit\"at Heidelberg, Im Neuenheimer Feld 205, 69120 Heidelberg, Germany}}

\email{pfhwangmath@gmail.com}

\subjclass[2010]{14D20, 14J60, 32G13,  53C07}
\keywords{$\lambda$-Flat Bundles, (Pluri-)harmonic Metrics, Simpson--Mochizuki Correspondence, Moduli Spaces, Dynamical System}
\date{}
\maketitle

\begin{abstract}
The notion of flat $\lambda$-connections as the interpolation of usual flat connections and Higgs fields was suggested by Deligne and  further  studied by Simpson. Mochizuki established  the Kobayashi--Hitchin-type theorem for $\lambda$-flat bundles ($\lambda\neq 0$), which is called the Mochizuki correspondence. In this paper, on the one hand, we generalize Mochizuki's result to the case when the base  being a compact balanced manifold, more precisely, we prove the existence of harmonic metrics on stable  $\lambda$-flat bundles ($\lambda\neq 0$). On the other hand, we study two applications of the Simpson--Mochizuki correspondence to moduli spaces. More concretely, we show this correspondence provides a homeomorphism between the moduli space of  (semi)stable $\lambda$-flat bundles over a complex projective manifold and the  Dolbeault moduli space, and also  provides dynamical systems with two parameters on the latter moduli space. We investigate such dynamical systems, in particular, we calculate the  first variation, the fixed points and discuss the asymptotic behaviour.
\end{abstract}

\renewcommand\abstractname{R\'esum\'e}
\begin{abstract}
La notion de $\lambda$-connexions plates comme interpolation de connexions plates habituelles et champs de Higgs a \'et\'e sugg\'er\'ee par Deligne et étudiée plus en d\'etail par Simpson. Mochizuki a \'etabli le th\'eor\`eme de type Kobayashi--Hitchin pour les fibr\'es $\lambda$-plats ($\lambda\neq0$), qui s'appelle la correspondence de Mochizuki. Dans cet article, d'une part, nous g\'en\'eralisons le r\'esultat de Mochizuki au cas o\`u la vari\'et\'e de base est une vari\'et\'e \'equilibr\'ee, plus pr\'ecis\'ement, nous prouvons l'existence de m\'etriques de harmoniques sur les fibr\'es $\lambda$-plats stables ($\lambda\neq0$). D'autre part, nous étudions deux applications de la correspondance de Simpson--Mochizuki aux espaces de modules. Plus concr\`etement, nous montrons que cette correspondance fournit un hom\'eomorphisme entre l'espace des modules des fibr\'es $\lambda$-plats (semi)stables sur une vari\'et\'e projective complexe et l'espace des modules de Dolbeault, et fournit \'egalement des systèmes dynamiques avec deux param\'etres sur ce dernier espace des modules. Nous \'etudions de tels syst\`emes dynamiques, en particulier, nous calculons la premi\`ere variation, les points fixes et discutons le comportement asymptotique.
\end{abstract}

\tableofcontents

\section{Introduction}

The notion of flat $\lambda$-connections as the interpolation of usual flat connections and Higgs fields was suggested by Deligne \cite{Del}, illustrated by Simpson in \cite{CS7} and further  studied in \cite{CS8,CS9}.
By applying Simpson's construction for the moduli space of $\Lambda$-modules \cite{CS5}, one can show the existence of the coarse moduli space of rank $r$ semistable $\lambda$-flat bundles with vanishing Chern classes over a complex projective manifold $X$, which is denoted by $\mathbb{M}_{\mathrm{Hod}}(X,r)$. And this construction can be generalized to  the case of principal bundles  by applying the Tannakian considerations \cite{CS7}. It is clear that $\mathbb{M}_{\mathrm{Hod}}(X,r)$ has a fibration over $\mathbb{C}$, in particular,  the fiber over $\lambda=0$ is the usual Dolbeault moduli space $\mathbb{M}_{\rm Dol}(X,r)$, and over $\lambda=1$ it is the usual de Rham moduli space $\mathbb{M}_{\rm dR}(X,r)$. Deligne's motivation is to understand Hitchin's twistor construction for the moduli space of solutions to Hitchin's self-duality equations which carries a hyperK\"{a}hler structure \cite{HKLR}. More precisely, according to Deligne's perspective, Hitchin's twistor space can be treated as the gluing of the  moduli space $\mathbb{M}_{\mathrm{Hod}}(X,r)$ and the complex conjugate moduli space $\mathbb{M}_{\mathrm{Hod}}(\bar{X},r)$ by the Riemann--Hilbert correspondence. Simpson interpreted the  moduli space $\mathbb{M}_{\mathrm{Hod}}(X,r)$ as the  Hodge filtration on the non-abelian de Rham cohomology $\mathbb{M}_{\rm dR}(X,r)$, and showed the Griffiths transversality and the regularity of the Gauss--Manin connection for this filtration \cite{CS7}. Since then, this notion attracts many researchers' attention, for example, flat $\lambda$-connections play a role in compactifying the de Rham moduli spaces \cite{CS7,IIS}; the author of \cite{Ari} used spectral curves to describe $\lambda$-connections that are formal deformations of Higgs bundles; and recently the authors of \cite{LRT} applied flat $\lambda$-connections to study the Kapustin--Witten equations.

The non-abelian Hodge correspondence provides  a homeomorphism $\mathbb{M}_{\rm Dol}(X,r)\simeq \mathbb{M}_{\rm dR}(X,r)$, which is a $C^\infty$-isomorphism over the smooth loci \cite{CS5}. This homeomorphism is achieved by finding the pluri-harmonic metrics, that is, by constructing the category of harmonic bundles in order to connect the Dolbeault side and the de Rham side. When $X$ is a compact K\"ahler manifold, such metrics exist for semisimple flat bundles due to Donaldson \cite{Don} and Corlette \cite{Cor}, and for polystable Higgs bundles with vanishing Chern classes due to Hitchin \cite{Hit} and Simpson \cite{CS1}. For $\lambda\neq0$, Mochizuki introduced the notion of pluri-harmonic metrics for  $\lambda$-flat bundles, and also established  the Kobayashi--Hitchin-type theorem for this case \cite{TM2}. We call this remarkable theorem the \emph{Mochizuki correspondence}. Moreover, together with  the Kobayashi--Hitchin correspondence for Higgs bundles, it is unified into the so called \emph{Simpson--Mochizuki correspondence}, which indicates the existence of pluri-harmonic metrics on $\lambda$-flat bundles  satisfying certain stability conditions for any $\lambda\in\mathbb{C}$.  By this correspondence, one can relate the category (moduli stack, moduli space) of polystable $\lambda$-flat bundles and that of polystable Higgs bundles.

\begin{introremark}
When $\lambda\neq 0$, by multiplying with $\lambda^{-1}$ reduces a stable $\lambda$-flat bundle  $(E, D^\lambda)$ $($see Definition \ref{def2.1}$)$ of rank $r$ to a usual flat bundle $(E, \lambda^{-1}D^\lambda)$, then by  Corlette's work, we also have a pluri-harmonic metric. However, this metric is very different from the metric given by the Mochizuki correspondence. For example, given a metric on $(E, \lambda^{-1}D^\lambda)$, there is a $\rho$-equivariant map $f:\tilde X\rightarrow \mathrm{GL}(r,\mathbb{C})/U(r)$, where $\tilde X$ is the universal cover of $X$, here $\rho:\pi_1(X)\rightarrow \mathrm{GL}(r,\mathbb{C})$ is a simple representation of the fundamental group $\pi_1(X)$ associated to $(E, \lambda^{-1}D^\lambda)$. Then in general, the energy $K_f=\int_X|df|^2d\nu_X$, where $d\nu_X$ is the volume element of the K\"ahler metric on $X$, corresponding to Corlette's metric is smaller than that to Mochizuki's metric \cite{Don,Cor,CS3}.
 In our opinion, the  Simpson--Mochizuki correspondence exhibits more natural interpolation between $\lambda=1$ and $\lambda=0$, for example, for a given polystable $\lambda_0$-flat bundle, we have a family\footnote{Such family is called a \emph{twistor line} or a \emph{preferred section} under the context of twistor theory \cite{CS7}.} of polystable $\lambda$-flat bundles (varying $\lambda$) such that they correspond to the same Higgs bundle  whenever $\lambda_0\in \mathbb{C}$.
\end{introremark}

This paper is a study of the Simpson--Mochizuki correspondence. It is organized as the follows. 

In Section \ref{sec2}, as a preliminary, we collect some basic materials, simple conclusions, and provide  an explicit example.

In Section \ref{sec3}, we discuss the Simpson--Mochizuki correspondence at various levels, including the Kobayashi--Hitchin version, categorical version, and moduli version. In particular, following Simpson's ideas in \cite{CS6,CS9}, we show the following theorem.

\begin{introtheorem}[Corollary \ref{homeo}]
Let $X$ be a complex projective manifold, and let $\mathbb{M}_{\mathrm{Hod}}^\lambda(X,r)$ be the fiber of the fibration $\mathbb{M}_{\mathrm{Hod}}(X,r)\to\mathbb{C}$ over $\lambda\in\mathbb{C}$, then the Simpson--Mochizuki correspondence provides a homeomorphism
$$
\mathbb{M}_{\mathrm{Hod}}^\lambda(X,r)\simeq \mathbb{M}_{\mathrm{Dol}}(X,r).
$$
\end{introtheorem}

In Section \ref{sec4}, we consider the Mochizuki correspondence under a more general framework. Our generalization includes two aspects.
\begin{itemize}
\item[$\bullet$] Firstly, the base manifold $X$  is relaxed to be a compact \emph{balanced manifold}, that is, the associated fundamental $(1,1)$-form $\omega$ satisfies the condition $d(\omega^{\dim_\mathbb{C}X-1})=0$. Obviously, this condition is weaker than the K\"ahler  condition $d\omega=0$, but stronger than the Gauduchon condition $\partial\bar{\partial}(\omega^{\dim_\mathbb{C}X-1})=0$. 
\item[$\bullet$] Secondly, the pluri-harmonicity condition for the Hermitian metrics on  $\lambda$-flat bundles is replaced by the harmonicity condition (see Definition \ref{def2}). Obviously, the latter one is weaker in general. Of course, when $X$ is exactly a K\"{a}hler manifold and $\lambda\neq 0$, then these two conditions are fully equivalent (see Proposition \ref{kk}). 
\end{itemize}

Via the standard method of continuity, we show the following theorem.

\begin{introtheorem}[Theorem \ref{pluri-har}]
Let $X$ be a compact balanced manifold, and  $((E,\bar\partial_E),{D}^\lambda)$ be a stable $\lambda$-flat bundle over $X$ ($\lambda\neq 0$), then there is a unique harmonic metric on $((E,\bar\partial_E),{D}^\lambda)$ up to constant scalars.
\end{introtheorem}

\begin{introremark}
It is known that when $\lambda=0$, the Kobayashi--Hitchin problem (i.e. the existence of harmonic metrics on stable Higgs bundles with vanishing the first Chern class) can be  solved for Gauduchon manifolds \cite{LT}. However, the condition of the flatness of $\lambda$-connection ($\lambda\neq0$) is a more rigid constraint than that of Higgs field, so generally one cannot expect the above theorem to be true for Gauduchon manifolds as the case of Higgs bundles.
\end{introremark}

The last section, i.e. Section \ref{sec5}, is devoted to an application of the Simpson--Mochizuki correspondence to Dolbeault moduli spaces. More concretely, combining the Simpson--Mochizuki correspondence and $\mathbb{C}^*$-actions on Hodge moduli spaces together, we construct dynamical systems with two parameters on Dolbeault moduli spaces. Here a dynamical system means a continuous  self-map $\psi_{(\lambda,t)}:\mathbb{M}_{\mathrm{Dol}}(X,r)\rightarrow \mathbb{M}_{\mathrm{Dol}}(X,r) $ with a pair $(\lambda,t)\in\mathbb{C}\times\mathbb{C}^*$ of parameters. We first study the local property of $\psi_{(\lambda,t)}$ by calculating the first variation. Next we consider the fixed points of this map. 

For a given Higgs bundle $u:=((E,\bar\partial_E),\theta)\in \mathbb{M}_{\mathrm{Dol}}(X,r)$, we define the set of stable parameters
$$
\mathcal{C}_u=\{(\lambda,t)\in \mathbb{C}\times\mathbb{C}^*: \psi_{(\lambda,t)}(u)=u\},
$$
and for a given pair $(\lambda,t)\in\mathbb{C}\times\mathbb{C}^*$ of parameters, we define the set of fixed points
$$
\mathfrak{Fix}_{(\lambda,t)}=\{u\in \mathbb{M}_{\mathrm{Dol}}(X,r): \psi_{(\lambda,t)}(u)=u\}.
$$
Our main results on this topic can be summarized as follows:

\begin{introtheorem}[Theorem \ref{9}, Corollary \ref{coro4}, Theorem \ref{cv}]\
\begin{enumerate}
  \item Let $X$ be a  Riemann surface, and let $u\in  \mathbb{M}_{\mathrm{Dol}}(X,r)$ represents a decoupled Higgs bundle with nontrivial Higgs field, then
$\mathbb{C}\times\{\mu_l^m, m=0,\cdots,l-1\}\subseteq \mathcal{C}_u\subseteq (\mathbb{C}\times\{\mu_l^m, m=0,\cdots,l-1\})\bigcup\{(\lambda,t)\in \mathbb{C}^*\times\mathbb{C}^*:|t||\lambda|^2=1,|t|\neq1, t=|t|\mu_{l'}^k, k=1,\cdots,l'-1\}$, where $\mu_l=e^{\frac{2\pi i}{l}}, \mu_{l'}=e^{\frac{2\pi i}{l'}}$ for some fixed positive integers $1\leq l\leq r, 2\leq l'\leq r$.
In particular, if the Higgs field satisfies $\Tr(\theta)\neq0$ at some point $x\in X$, then $\mathcal{C}_u=\mathbb{C}\times \{1\}$.
  \item Let $\mathfrak{Fix}=\bigcap\limits_{(\lambda,t)\in \mathbb{C}^*\times \mathbb{C}^*}\mathfrak{Fix}_{(\lambda,t)}$. Then $\mathfrak{Fix}$ consists of the set of complex variations of Hodge structure\footnote{In this paper, we agree with the terminology of \cite{CW}, namely a complex variation of Hodge structure means a (polystable) system of Hodge bundles in the sense of Simpson's paper \cite{CS4}.}.
\end{enumerate}

\end{introtheorem}

Finally, to investigate the limiting behaviour of this dynamical system  when the parameters tend to 0, we introduce the following five limits of a Higgs bundle $((E,\bar\partial_E),\theta)\in \mathbb{M}_{\mathrm{Dol}}(X,r)$ (now $X$ is a Riemann surface):
\begin{enumerate}
\item[(1)] $\psi_{\underline{(0,0)}}((E,\bar\partial_E),\theta)
   :=\lim\limits_{t\rightarrow 0}\psi_{(0,t)}((E,\bar\partial_E),\theta),$
   \smallskip
   \item[(2)] $\psi^{\underline{(0,0)}}((E,\bar\partial_E),\theta)
   :=\lim\limits_{t\rightarrow 0}\lim\limits_{\lambda\rightarrow 0}\psi_{(\lambda,t)}((E,\bar\partial_E),\theta),$
   \smallskip
   \item[(3)] $\psi_{\overline{(0,0)}}((E,\bar\partial_E),\theta)
   :=\lim\limits_{\lambda\rightarrow 0}\psi_{(\lambda,0)}((E,\bar\partial_E),\theta)$,
   \smallskip
   \item[(4)] $\psi^{\overline{(0,0)}}((E,\bar\partial_E),\theta)
   :=\lim\limits_{\lambda\rightarrow 0}\lim\limits_{t\rightarrow 0}\psi_{(\lambda,t)}((E,\bar\partial_E),\theta),$
   \smallskip
   \item[(5)] $\psi_{(0,0)}((E,\bar\partial_E),\theta)
   :=\lim\limits_{(\lambda,t)\rightarrow (0,0)}\psi_{(\lambda,t)}((E,\bar\partial_E),\theta),$
\end{enumerate}
where $\psi_{(\lambda,0)}$ appeared in the third limit is defined by the Simpson filtration that is closely related to the limits of $\mathbb{C}^*$-action on $\mathbb{M}_{\mathrm{Hod}}(X,r)$. For a general Higgs bundle, it's quite hard to explicitly describe these limits, we do not even know whether they exist. However, if these limits exist, all are the complex variations of Hodge structure.  For some special cases, we discuss these limits.

\begin{introtheorem}[Theorem \ref{nb}]
Let $X$ be a Riemann surface.
\begin{enumerate}
 \item If $((E,\bar\partial_E),\theta))\in \mathbb{M}_{\mathrm{Dol}}(X,r)$ is a complex variation of Hodge structure or a decoupled Higgs bundle, then the above  limits exist and coincide.
\item Let $(E,\bar\partial_E),\theta))\in \mathbb{M}_{\mathrm{Dol}}(X,2)$ and  assume the maximal destabilizing subbundle of $(E,\bar\partial_E)$ is preserved by $\theta^\dagger_h$ for the pluri-harmonic metric $h$ on $((E,\bar\partial_E),\theta))$, then  the limit $ \psi_{\overline{(0,0)}}((E,\bar\partial_E),\theta)$ exists,  and it  coincides with the limit $ \psi_{\underline{(0,0)}}((E,\bar\partial_E),\theta)$.
\item Let  $(E,\bar\partial_E),\theta))\in M_{\mathrm{Dol}}(X,r)$, then the limit  $\lim\limits_{\lambda\rightarrow 0}\psi_{(\lambda,0)}((E,\bar\partial_E),\lambda\theta)$ exists, and it  coincides with the limit $ \psi_{\underline{(0,0)}}((E,\bar\partial_E),\theta)$.
\end{enumerate}

\end{introtheorem}

\noindent\textbf{Acknowledgements}. 
The author P. Huang would like to thank his thesis supervisor Prof. Carlos Simpson for the kind help and useful discussions. Both authors would like to thank Prof. Takuro Mochizuki, Prof. Kang Zuo and Dr. Ya Deng for their useful discussions on various occasions. 

 \section{Preliminaries }\label{sec2}
 
\subsection{Flat $\lambda$-Connections, Pluri-harmonic Metrics}
\begin{definition} [\cite{CS7,TM2}]\label{def2.1}
Let $X$ be a complex projective manifold and $E$ be a holomorphic vector bundle over $X$, with the underlying smooth vector bundle denoted by $\mathbb{E}$. Fix $\lambda\in\mathbb{C}$.
\begin{enumerate}
  \item A \emph{holomorphic $\lambda$-connection} on $E$ is a $\mathbb{C}$-linear map $D^\lambda: E\to E\otimes\Omega_X^{1}$ that satisfies the following $\lambda$-twisted Leibniz rule:
$$D^\lambda(fs)=fD^\lambda s+\lambda s\otimes  df,$$
where $f$ and $s$ are holomorphic sections of $\mathcal{O}_X$ and $E$, respectively. It naturally extends to a map $D^\lambda: E\otimes\Omega_X^{p}\to E\otimes\Omega_X^{p+1}$ for any integer $p\geq0$. If $D^\lambda\circ D^\lambda=0$, we call $D^\lambda$ a (holomorphic) \emph{flat $\lambda$-connection} and the pair $(E,D^\lambda)$ is called a (holomorphic) \emph{$\lambda$-flat bundle}.
  \item A \emph{$C^\infty$ $\lambda$-connection} on $\mathbb{E}$ is a $\mathbb{C}$-linear map $\mathbb{D}^\lambda: \mathbb{E}\to\mathbb{E}\otimes T^*X$ that satisfies the following $\lambda$-twisted Leibniz rule:
$$\mathbb{D}^\lambda(fs)=f\mathbb{D}^\lambda s+\lambda s\otimes\partial f+s\otimes\bar{\partial}f,$$
where $f$ is a smooth function on $X$ and $s$ is a smooth section of $\mathbb{E}$. It naturally extends to a map $\mathbb{D}^\lambda: \mathbb{E}\otimes\Lambda^r(T^*X)\to\mathbb{E}\otimes\Lambda^{r+1}(T^*X)$ for any integer $r\geq0$. If $\mathbb{D}^\lambda\circ\mathbb{D}^\lambda=0$, we call $\mathbb{D}^\lambda$ a  ($C^\infty$) \emph{flat $\lambda$-connection}, and the pair $(\mathbb{E},\mathbb{D}^\lambda)$ is called a  ($C^\infty$) \emph{$\lambda$-flat bundle}.
\end{enumerate}
\end{definition}

\begin{remark}
Obviously, when $\lambda=1$ and $0$, then above definition reduces to that of a usual flat connection and Higgs field, respectively.
Giving a holomorphic flat $\lambda$-connection $D^\lambda$ on $E$ is equivalent to giving  a $C^\infty$ flat $\lambda$-connection $\mathbb{D}^\lambda$ on $\mathbb{E}$.
For simplicity,  we do not distinguish  $E$  and  $\mathbb{E}$ when there is no ambiguity, and for a $\lambda$-flat bundle, we have various notations such as $(E, D^\lambda), (E,\mathbb{D}^\lambda), ((E,\bar\partial_E),D^\lambda)$ or $ ((E,d^{\prime\prime}_E),d^\prime_E),$ depending on the different contexts. Additionally,  above notions can also work for the category of coherent sheaves, i.e.  $\lambda$-flat bundles can be generalized to \emph{$\lambda$-flat coherent sheaves} without any difficulty.
\end{remark}

Now we consider the $\lambda$-connections in the $C^\infty$-category for   more general base manifold $X$, namely we assume $X$ is a compact balanced manifold.
Fixing  $\lambda\in \mathbb{C}$,  let $(E,\mathbb{D}^\lambda)$ be a $\lambda$-flat bundle over $X$,  and let $h$ be a Hermitian metric on $E$. We decompose $\mathbb{D}^\lambda$ into its (1,0)-part $d_E'$ and (0,1)-part $d_E''$ that
 defines a holomorphic structure on $E$.
From $h$ and $d_E'$, we have a (0,1)-operator $\delta_h''$ determined by the condition $\lambda\partial h(u,v)=h(d_E'u,v)+h(u,\delta_h''v)$,
similarly, $h$ and $d_E''$ provides  a (1,0)-operator $\delta_h'$ via the  condition
$\bar{\partial}h(u,v)=h(d_E''u,v)+h(u,\delta_h'v)$.
One easily checks that
$\delta_h'(fv)=f\delta'v+v\otimes\partial f$, and
$\delta_h''(fv)=f\delta_h''v+\bar{\lambda}v\otimes\bar{\partial}f$.
We introduce  the following four operators
\begin{equation}
\begin{aligned}
\partial_h&:=\frac{1}{1+|\lambda|^2}\bigg(\bar{\lambda}d_E'+\delta_h'\bigg),\quad
\bar{\partial}_h:=\frac{1}{1+|\lambda|^2}\bigg(d_E''+\lambda\delta_h''\bigg),\\
\theta_h&:=\frac{1}{1+|\lambda|^2}\bigg(d_E'-\lambda\delta_h'\bigg),\quad
\theta_h^\dagger:=\frac{1}{1+|\lambda|^2}\bigg(\bar{\lambda}d_E''-\delta_h''\bigg).
\end{aligned}
\end{equation}
They satisfy
\begin{equation}
\begin{aligned}
d_E'&=\lambda\partial_h+\theta_h,\quad
d_E''=\bar{\partial}_h+\lambda\theta_h^\dagger,\\
\delta_h'&=\partial_h-\bar{\lambda}\theta_h,\quad
\delta_h''=\bar{\lambda}\bar{\partial}_h-\theta_h^\dagger.
\end{aligned}
\end{equation}
Now  $\partial_h$ and $\bar{\partial}_h$ obey  the usual Leibniz rule,
 $\theta_h\in C^\infty(X,\Omega_X^{1,0}\otimes\End(E))$ and $\theta_h^\dagger\in C^\infty(X,\Omega_X^{0,1}\otimes\End(E))$. Moreover, it's easy to check that $\mathcal{D}_h :=\partial_h+\bar{\partial}_h$, $d_E''+\delta_h'$ and $\lambda^{-1}d_E'+\bar{\lambda}^{-1}\delta_h''$ $(\lambda\neq0)$ are  unitary connections with respect to the metric $h$, and $\theta_h^\dagger$ is the adjoint of $\theta_h$ in the sense that
$$
h(\theta_h(u),v)=h(u,\theta_h^\dagger(v)).
$$ 
We also introduce the operators ${\mathbb{D}_h^\lambda}^\star=\delta_h'-\delta_h''$ and $G(h,\mathbb{D}^\lambda )=[\mathbb{D}^\lambda,{\mathbb{D}_h^\lambda}^\star]$, the latter one is called the \emph{pseudo-curvature}.

\begin{definition}\label{def2}
The Hermitian metric $h$  on a $\lambda$-flat bundle $(E,\mathbb{D}^\lambda)$ is called
\begin{enumerate}
  \item a \emph{harmonic metric} if $\Lambda_\omega G(h,\mathbb{D}^\lambda)=0$, where  $\Lambda_\omega$ stands for the contraction by $\omega$,
  \item a \emph{pluri-harmonic metric} if $G(h,\mathbb{D}^\lambda)=0$.
\end{enumerate}
\end{definition}

\begin{proposition}[\textbf{K\"{a}hler Identities of Flat $\lambda$-Connections}, \cite{TM2}]\label{ki}
Let $(X,\omega)$ be a  compact K\"{a}hler manifold, then we have
\begin{align*}
 (\mathbb{D}^\lambda)^*_{h}&=-\sqrt{-1}[\Lambda_\omega,{\mathbb{D}^\lambda_{h}}^\star],\\
 ({\mathbb{D}^\lambda_{h}}^\star)^*_{h}&=\sqrt{-1}[\Lambda_\omega,{\mathbb{D}^\lambda}].
\end{align*}

\end{proposition}

The following property says, for a Hermitian metric on a $\lambda$-flat bundle ($\lambda\neq0$) over a compact K\"ahler manifold, it is a pluri-harmonic metric if and only if it is  a harmonic metric, or if and only if the $(1,1)$-part of its pseudo-curvature vanishes.

\begin{proposition}\label{kk}
 Let $\lambda\neq 0$, and let $(X,\omega)$ be a  compact K\"{a}hler manifold, then all the following  conditions are equivalent:
 \begin{enumerate}
\item    $G(h,\mathbb{D}^\lambda )=0$,
\item $\Lambda_\omega G(h,\mathbb{D}^\lambda )=0$,
\item $(\bar\partial_h+\theta_h)^2=0$,
\item $(\partial_h+\theta_h^\dagger)^2=0$,
\item $\widetilde{\bar\partial_h}\theta_h=0$\footnote{Here we add the notation $\ \tilde{}\ $ to indicate the induced operator on $\End(E)\otimes \Omega^{\bullet,\bullet}_X$ from the operator on $E\otimes \Omega^{\bullet,\bullet}_X$.} and $\theta_h^2=0$,
\item $\widetilde{\partial_h}\theta^\dagger_h=0$ and $(\theta_h^\dagger)^2=0$,
\item $\widetilde{\bar\partial_h}\theta_h=0$,
\item $\widetilde{\partial_h}\theta^\dagger_h=0$,
\item $\Lambda_\omega\widetilde{\bar\partial_h}\theta_h=0$,
\item $\Lambda_\omega\widetilde{\partial_h}\theta^\dagger_h=0$.
\end{enumerate}
\end{proposition}

\begin{proof}
We only give the sketch of the proof of  $(1)\Leftrightarrow (2)$, namely $h$ is a pluri-harmonic metric if and only if it is a harmonic metric,  more details can be found in the second named author's thesis \cite{Hua}. The equivalence of (1), (3), (4), (5), (6) has been shown in \cite{TM2}. And the equivalence of (5), (6), (7), (8) is recently proved by  Mochizuki in \cite{TM4}. By the flatness of $\mathbb{D}^\lambda$, we have $({\mathbb{D}_h^\lambda}^\star)^2=0$, which yields the following   Bianchi identities
$$
\widetilde{\mathbb{D}^\lambda} G(h,\mathbb{D}^\lambda )=\widetilde{{\mathbb{D}_h^\lambda}^\star}G(h,\mathbb{D}^\lambda )=0.
$$
Therefore, it follows from the identity
$$
G(h,\mathbb{D}^\lambda )=\frac{\lambda}{1+|\lambda|^2}\widetilde{\mathbb{D}^\lambda}(\lambda^{-1}\theta_h-\theta_h^\dagger)
$$
and the assumption $\Lambda_\omega G(h,\mathbb{D}^\lambda )=0$ that
\begin{align*}
  \int_X\langle G(h,\mathbb{D}^\lambda ),G(h,\mathbb{D}^\lambda )\rangle_{h,\omega}\omega^n= \frac{\lambda}{1+|\lambda|^2}\int_X\langle \lambda^{-1}\theta_h-\theta_h^\dagger,(\widetilde{\mathbb{D}^\lambda})^*G(h,\mathbb{D}^\lambda )\rangle_{h,\omega}\omega^n=0,
\end{align*}
thus $G(h,\mathbb{D}^\lambda )=0$.
\end{proof}

\begin{remark}
Very recently, the authors of \cite{ChWe} introduced $n$-dimensional balanced manifolds of Hodge--Riemann type, namely imposing a further condition
$$
\frac{\omega^{n-1}}{(n-1)!}=\omega_0\wedge\Omega_0
$$
for certain real $(1,1)$-form $\omega_0$ and $(n-2,n-2)$-form $\Omega_0$ satisfying the Hodge--Riemann bilinear relation. For such special balanced manifolds, the above proposition still holds (cf. \cite[Proposition 2.15]{TM4} and \cite[Theorem 5.1]{ChWe}).
\end{remark}

\begin{proposition}\label{w}
Let $\lambda\neq 0$, and let $(E,\mathbb{D}^\lambda)$ be a $\lambda$-flat bundle over a  Riemann surface $(X,\omega)$ together with a Hermitian metric $h$, then
\begin{enumerate}
  \item for any local  $\mathbb{D}^\lambda$-flat section $s$ of $E$, we have
$$
\Delta_\omega|s|_h^2\geq-\frac{2}{1+|\lambda|^2}|s|^2_h|\Lambda_\omega G(h,\mathbb{D}^\lambda) |_h,
$$
where $\Delta_\omega$ denotes the usual Laplacian on $(X,\omega)$.
  \item for any local nowhere-vanishing $\mathbb{D}^\lambda$-flat section $s$ of $E$, we have
$$
\Delta_\omega\log(|s|_h^2)\geq -\frac{2}{1+|\lambda|^2}|\Lambda_\omega G(h,\mathbb{D}^\lambda)|_h.
$$
\end{enumerate}

\end{proposition}
\begin{proof} 
(1) Let  $s$ be a local  $\mathbb{D}^\lambda$-flat section, namely
 we have
$$
\begin{aligned}
d_E's&=(\lambda\partial_h+\theta_h)s=0,\\
d_E''s&=(\bar{\partial}_h+\lambda\theta_h^\dagger)s=0,
\end{aligned}
$$
then
\begin{align*}
  \bar{\partial}h(s,s)&=h(d_E''s,s)+h(s,\delta_h's)=h(s,\delta_h's),\\
 \lambda {\partial}h(s,s)&=h(d_E's,s)+h(s,\delta_h''s)=h(s,\delta_h''s),
\end{align*}
which gives rise to
\begin{align*}
\lambda\partial\bar{\partial}h(s,s)=\lambda\partial h(s,\delta_h's)=h(d_E's,\delta_h's)+h(s,\delta_h''\delta_h's)=h(s,\delta_h''\delta_h's).
\end{align*}
By means of the following identities
\begin{align*}
h(s,\bar{\lambda}\bar{\partial}_h\partial_h(s))&=-h(s,\bar{\lambda}\bar{\partial}_h\frac{\theta_h}{\lambda}(s))=
-h(s,\frac{\bar{\lambda}}{\lambda}(\widetilde{\bar{\partial}_h}\theta_h)(s))-h(s,\bar{\lambda}\theta_h\theta_h^\dagger(s))\\
&=-\lambda|\theta_h^\dagger(s)|_h^2+\frac{|\lambda|^2}{\bar \lambda(1+|\lambda|^2)^2}h(s, G(h,\mathbb{D}^\lambda)s),\\
h(s,\theta_h^\dagger\partial_h(s))&=-h(s,\theta_h^\dagger\frac{\theta_h}{\lambda}(s))=-\frac{1}{\bar{\lambda}}|\theta_h(s)|_h^2,\\
h(s,\bar{\lambda}^2\bar{\partial}_h\theta_h(s))&=h(s,\bar{\lambda}^2(\widetilde{\bar{\partial}_h}\theta_h)(s))+h(s,(\bar{\lambda})^2\theta_h\lambda\theta_h^\dagger(s))\\
&=|\lambda|^2\lambda|\theta_h^\dagger(s)|_h^2-\frac{\lambda|\lambda|^2}{(1+|\lambda|^2)^2}h(s, G(h,\mathbb{D}^\lambda)s),
\end{align*}
 we obtain\begin{align*}
h(s,\delta_h''\delta_h'(s))&=h(s,(\bar{\lambda}\bar{\partial}_h-\theta_h^\dagger)(\partial_h-\bar{\lambda}\theta_h)(s))\\
&=h(s,\bar{\lambda}\bar{\partial}_h\partial_h(s))-h(s,\theta_h^\dagger\partial_h(s))+h(s,\bar{\lambda}\theta_h^\dagger\theta_h(s))-h(s,\bar{\lambda}^2\bar{\partial}_h\theta_h(s))\\
&=(1+|\lambda|^2)(-\lambda|\theta_h^\dagger(s)|_h^2+\frac{1}{\bar{\lambda}}|\theta_h(s)|_h^2)+\frac{\lambda}{1+|\lambda|^2}h(s, G(h,\mathbb{D}^\lambda)s).
\end{align*}
It follows that
\begin{align*}
  -\Delta_\omega|s|_h^2=&\ 2\sqrt{-1}\Lambda_\omega\partial\bar{\partial}|s|_h^2\\
  &=2\sqrt{-1}\Lambda_\omega[-(1+|\lambda|^2)|\theta_h^\dagger(s)|_h^2+\frac{1+{|\lambda|^2}}{|\lambda|^2}|\theta_h(s)|_h^2+\frac{1}{1+|\lambda|^2}h(s, G(h,\mathbb{D}^\lambda)s)]\\
  &\leq-\frac{2}{1+|\lambda|^2} h(s,\sqrt{-1} \Lambda_\omega G(h,\mathbb{D}^\lambda)s)\\
  &\leq\frac{2}{1+|\lambda|^2}|s|_h|\Lambda_\omega G(h,\mathbb{D}^\lambda) s|_h\leq\frac{\lambda}{1+|\lambda|^2}|s|^2_h|\Lambda_\omega G(h,\mathbb{D}^\lambda) |_h,
\end{align*}
where we apply the Cauchy--Schwarz inequality for the last two inequalities.

(2)
We have
\begin{align*}
\lambda\partial\bar{\partial}\log(|s|_h^2)&=\frac{\lambda\partial\bar{\partial}|s|_h^2}{|s|_h^2}-\frac{\lambda\partial|s|_h^2\wedge \bar{\partial}|s|_h^2}{|s|^4_h}\\
&=\frac{h(s,\delta_h''\delta_h's)}{|s|_h^2}-\frac{h(s,\delta_h''s)\wedge h(s,\delta_h's)}{|s|_h^4},
\end{align*}
where the first term on the right hand side of  the second equality has been calculated, and the second term can be calculated by
the identities
\begin{align*}
h(s,\delta_h'(s))&=h(s,(\partial_h-\bar{\lambda}\theta_h)(s))=h(s,(-\lambda^{-1}-\bar{\lambda})\theta_h(s))=-\frac{(1+|\lambda|^2)}{\bar{\lambda}}h(s,\theta_h(s)),\\
h(s,\delta_h''(s))&=h(s,(\bar{\lambda}\bar{\partial}_h-\theta_h^\dagger)(s))=-(1+|\lambda|^2)h(s,\theta_h^\dagger(s)).
\end{align*}
Finally, we arrive at
\begin{align*}
-\Delta_\omega\log(|s|_h^2)=&\ 2\sqrt{-1}\Lambda_\omega\partial\bar{\partial}\log(|s|_h^2)\\
=&\ 2\sqrt{-1}\Lambda_\omega[-(1+|\lambda|^2)\frac{|\theta_h^\dagger(s)|_h^2}{|s|_h^2}+\frac{1+|\lambda|^2}{|\lambda|^2}\frac{|\theta_h(s)|_h^2}{|s|_h^2}\\
&\ -\frac{(1+|\lambda|^2)^2}{|\lambda|^2}\frac{h(s,\theta_h(s))\wedge h(s,\theta_h^\dagger(s))}{|s|_h^4}+\frac{2}{1+|\lambda|^2}\frac{h(s, G(h,\mathbb{D}^\lambda)s)}{|s|^2_h}]\\
\leq&\ \frac{2}{1+|\lambda|^2}|\Lambda_\omega G(h,\mathbb{D}^\lambda) |_h.
\end{align*}
We complete the proof.
\end{proof}

\subsection{Example}
Let   $E$ be a  Hermitian vector bundle over the punctured unit disk $\bigtriangleup^*=\{z:0<|z|<1\}$ of rank 2 with the  local unitary frame $\{v_1,v_2\}$. In \cite{MSWW}, the authors introduced  the so-called ``fiducial solution" of Hitchin's equations expressed in terms of the frame $\{v_1,v_2\}$ as follows
$$
\begin{aligned}
A&=\frac{1}{8}\begin{pmatrix}1&0\\0&-1\end{pmatrix}\left(\frac{dz}{z}-\frac{d\bar{z}}{\bar{z}}\right)\\
\theta&=\begin{pmatrix}0&\sqrt{|z|}\\\frac{z}{\sqrt{|z|}}&0\end{pmatrix}dz,
\end{aligned}
$$
 that solves the decoupled Hitchin's equations
$$
F_A=0,\quad [\theta,\theta^\dagger]=0,\quad \bar{\partial}_A\theta=0,
$$
where  $F_A$ denotes  the curvature of the connection $A$, and $\theta$ is the Higgs field.  Let $\mu\in\mathbb{C}^*$ be a constant, then we have a flat $\lambda$-connection $\mathbb{D}^\lambda_\mu=d_E'+d_E''$ with
$$
d_E'=\lambda\partial_A+\theta,\ \
d_E''=\bar{\partial}_A+\mu\theta^\dagger,$$
then a $\mathbb{D}^\lambda_\mu$-flat section $s=\begin{pmatrix}
f(z,\bar{z})\\g(z,\bar{z})\end{pmatrix}$ should satisfy the following equations
\begin{equation}\label{2.2}
\left\{\begin{aligned}
\lambda\frac{\partial f}{\partial z}+\frac{\lambda}{8}\frac{f}{z}+\sqrt{|z|}g&=0,\\
\lambda\frac{\partial g}{\partial z}-\frac{\lambda}{8}\frac{g}{z}+\frac{z}{\sqrt{|z|}}f&=0,\\
\frac{\partial f}{\partial\bar{z}}-\frac{1}{8}\frac{f}{\bar{z}}+\mu\frac{\bar{z}}{\sqrt{|z|}}g&=0,\\
\frac{\partial g}{\partial\bar{z}}+\frac{1}{8}\frac{g}{\bar{z}}+\mu\sqrt{|z|}f&=0.
\end{aligned}
\right.
\end{equation}
Let $z\to0$ in equations (\ref{2.2}), we have
$$
f(z,\bar{z})\longrightarrow z^{-\frac{1}{8}}\bar{z}^{\frac{1}{8}},\ \
g(z,\bar{z})\longrightarrow z^{\frac{1}{8}}\bar{z}^{-\frac{1}{8}},
$$
hence we can assume that
$$
f(z,\bar{z})=z^{-\frac{1}{8}}\bar{z}^{\frac{1}{8}}u(z,\bar{z}),\ \
g(z,\bar{z})= z^{\frac{1}{8}}\bar{z}^{-\frac{1}{8}}v(z,\bar{z})
$$
with $\lim\limits_{z\to0}u=\lim\limits_{z\to0}v=1$.
Then $u(z,\bar{z})$ and $v(z,\bar{z})$ should satisfy the following equations
\begin{equation*}
\left\{\begin{aligned}
\lambda\frac{\partial u}{\partial z}+\sqrt{z}v=0,\\
\lambda\frac{\partial v}{\partial z}+\sqrt{z}u=0,\\
\frac{\partial u}{\partial\bar{z}}+\mu\sqrt{\bar{z}}v=0,\\
\frac{\partial v}{\partial\bar{z}}+\mu\sqrt{\bar{z}}u=0,
\end{aligned}
\right.
\end{equation*}
which imply
$$\frac{\partial u}{\partial(\frac{z^{\frac{3}{2}}}{\lambda\mu})}=\frac{\partial u}{\partial(\bar{z}^{\frac{3}{2}})},\ \ \frac{\partial v}{\partial(\frac{z^{\frac{3}{2}}}{\lambda\mu})}=\frac{\partial v}{\partial(\bar{z}^{\frac{3}{2}})}.$$
Therefore, we can write
$$
u(z,\bar{z})=U\bigg(\frac{z^{\frac{3}{2}}}{\lambda\mu}+\bar{z}^{\frac{3}{2}}\bigg),\quad  
v(z,\bar{z})=V\bigg(\frac{z^{\frac{3}{2}}}{\lambda\mu}+\bar{z}^{\frac{3}{2}}\bigg).
$$ 
Introducing the new variable  $X=\frac{z^{\frac{3}{2}}}{\lambda\mu}+\bar{z}^{\frac{3}{2}}$,  we have
$$\left\{
\begin{aligned}
\frac{3}{2\mu}\frac{\partial U}{\partial X}+V&=0,\\
\frac{3}{2\mu}\frac{\partial V}{\partial X}+U&=0,\\
\end{aligned}\right.
$$
which can be solved easily
$$\begin{aligned}
U(X)&=C_1\exp(\frac{2\mu}{3}X)+C_2\exp(-\frac{2\mu}{3}X),\\
V(X)&=-C_1\exp(\frac{2\mu}{3}X)+C_2\exp(-\frac{2\mu}{3}X),
\end{aligned}$$
where $C_1$ and $C_2$ are two constants. Consequently, any  local $\mathbb{D}^\lambda_\mu$-flat section $s$  is the  $\mathbb{C}$-linear combination of the following two sections
\begin{align*}
  s_1&=\begin{pmatrix}
z^{-\frac{1}{8}}\bar{z}^{\frac{1}{8}}\exp(\frac{2}{3\lambda}z^{\frac{3}{2}}+\frac{2\mu}{3}\bar{z}^{\frac{3}{2}})\\
-z^{\frac{1}{8}}\bar{z}^{-\frac{1}{8}}\exp(\frac{2}{3\lambda}z^{\frac{3}{2}}+\frac{2\mu}{3}\bar{z}^{\frac{3}{2}})
\end{pmatrix},\\
s_2&=\begin{pmatrix}
z^{-\frac{1}{8}}\bar{z}^{\frac{1}{8}}\exp(-\frac{2}{3\lambda}z^{\frac{3}{2}}-\frac{2\mu}{3}\bar{z}^{\frac{3}{2}})\\
z^{\frac{1}{8}}\bar{z}^{-\frac{1}{8}}\exp(-\frac{2}{3\lambda}z^{\frac{3}{2}}-\frac{2\mu}{3}\bar{z}^{\frac{3}{2}})
\end{pmatrix}.
\end{align*}
One easily checks that $\Delta\log(|s|_h^2)=0$.

 Now let $\lambda'=t\lambda$.  We want to find the pluri-harmonic metric $h_t$ for the $\lambda'$-flat bundle $(E,\mathbb{D}^{\lambda^\prime}=td^\prime_E+d^{\prime\prime}_E)$ with $\mu=\lambda$.  Denote the matrix form of $h_t$ in terms of the frame $\{v_1,v_2\}$ by $H_t$.
We  write $t\cdot d_E'=\lambda'\partial_A+t\theta,
d_E''=\bar{\partial}_A+\lambda(1-|t|^2)\theta^\dagger+\lambda'(t\theta)^\dagger$, then one can take $((E,\bar{\partial}_A+\lambda(1-|t|^2)\theta^\dagger), t\theta)$ as the Higgs bundle by requiring
$$\left\{
\begin{aligned}
\bar{H}_t^{-1}\theta^\dagger\bar{H}_t&=\theta^\dagger,\\
\partial H_t&=(A^{1,0})^T\ H_t+H_t\overline{(A^{0,1}+\lambda(1-|t|^2)\theta^\dagger)},\\
\bar{\partial}H_t&=(A^{0,1}+\lambda(1-|t|^2\theta^\dagger)^TH_t+H_t\overline{A^{1,0}}.
\end{aligned}\right.$$
 Expressing $H_t$ as
$$
H_t=\begin{pmatrix}
a(z,\bar{z})&b(z,\bar{z})\\\overline{b(z,\bar{z})}&c(z,\bar{z})
\end{pmatrix},
$$
 then we have
$$\left\{
\begin{aligned}
a&=c,\\
b\cdot\bar{z}^{\frac{1}{2}}&=\bar{b}\cdot z^{\frac{1}{2}},\\
\frac{\partial a}{\partial z}&=\bar{\lambda}(1-|t|^2)b\sqrt{|z|},\\
\frac{\partial a}{\partial\bar{z}}&=\lambda(1-|t|^2)b\frac{\bar{z}}{\sqrt{|z|}},\\
\frac{\partial b}{\partial z}&=\frac{b}{4z}+\bar{\lambda}(1-|t|^2)a\frac{z}{\sqrt{|z|}},\\
\frac{\partial b}{\partial\bar{z}}&=-\frac{b}{4\bar{z}}+\lambda(1-|t|^2)a\sqrt{|z|}.
\end{aligned}\right.$$
It can be resolved as follows
$$
\begin{aligned}
a(z,\bar{z})&=f(\bar{z})\exp\bigg(\frac{2}{3}\bar{\lambda}(1-|t|^2)z^{\frac{3}{2}}\bigg)+g(\bar{z})\exp\bigg(-\frac{2}{3}\bar{\lambda}(1-|t|^2)z^{\frac{3}{2}}\bigg),\\
b(z,\bar{z})&=\frac{z^{\frac{1}{2}}}{\sqrt{|z|}}\left(f(\bar{z})\exp\bigg(\frac{2}{3}\bar{\lambda}(1-|t|^2)z^{\frac{3}{2}}\bigg)-g(\bar{z})\exp\bigg(-\frac{2}{3}\bar{\lambda}(1-|t|^2)z^{\frac{3}{2}}\bigg)\right),
\end{aligned}$$
where
$$
\begin{aligned}
f(\bar{z})&=C_1\exp\bigg(\frac{2}{3}\lambda(1-|t|^2)\bar{z}^{\frac{3}{2}}\bigg)+C_2\exp\bigg(-\frac{2}{3}\lambda(1-|t|^2)\bar{z}^{\frac{3}{2}}\bigg),\\
g(\bar{z})&=C_2\exp\bigg(\frac{2}{3}\lambda(1-|t|^2)\bar{z}^{\frac{3}{2}}\bigg)+C_3\exp\bigg(-\frac{2}{3}\lambda(1-|t|^2)\bar{z}^{\frac{3}{2}}\bigg),
\end{aligned}$$
for constants $C_1, C_2$ and $C_3$.

\begin{remark}This example exhibits  the non-uniqueness of pluri-harmonic metrics on $\lambda$-flat bundles over a non-complete manifold.
\end{remark}

\section{Simpson--Mochizuki Correspondence}\label{sec3}

\subsection{Categorical Version}

\begin{definition}[\cite{TM2,TM4}]\
\begin{enumerate}
  \item Let $X$ be a complex projective manifold with a fixed ample line bundle $L$. A  $\lambda$-flat bundle $(E,D^\lambda)$ over $X$ is called \emph{$\mu_L$-stable} (resp. \emph{$\mu_L$-semistable})\footnote{ Sometimes we omit the notation $\mu_L$ when there is no ambiguity.} if for any $\lambda$-flat subbundle $(V,D^\lambda|_V)$ of $0<\rank(V)<\rank(E)$, we have the following inequality
  $$
  \mu_L(V)< \mu_L(E)\ (\text{resp. } \mu_L(V)\leq \mu_L(E)),
  $$
where $\mu_L(\bullet)=\frac{\deg(\bullet)}{\rank(\bullet)}$ denotes the slope of bundle with respect to $L$. It is \emph{$\mu_L$-polystable}\footnote{When $\lambda\neq0$, a $\lambda$-flat bundle $(E,\mathbb{D}^\lambda)$ is $\mu_L$-stable  if and only if it is simple, namely it has no non-trivial proper $\lambda$-flat subbundle, and $(E,\mathbb{D}^\lambda)$ is $\mu_L$-polystable  if and only if it is semisimple, namely it is a direct sum of simple $\lambda$-flat bundles. } if it decomposes as a direct sum of $\mu_L$-stable $\lambda$-flat bundles with the same slope.
  \item Let $X$ be an $n$-dimensional compact K\"{a}hler manifold with a K\"{a}hler form $\omega$. A  $\lambda$-flat bundle $(E,\mathbb{D}^\lambda)$ with a Hermitian metric $h$  over $X$ is called \emph{analytically stable} (resp. \emph{analytically semistable}) if for any $\lambda$-flat torsion-free coherent subsheaf $(V,\mathbb{D}^\lambda|_V)$ of $0<\rank(V)<\rank(E)$, we have the following inequality
  $$
  \mu_\omega(V)< \mu_\omega(E)\ \ (\text{resp. } \mu_\omega(V)\leq \mu_\omega(E)),
  $$
  where $\mu_\omega(\bullet)=\frac{\int_{X\backslash S(\bullet)} \Tr(G(h|_\bullet,\mathbb{D}^\lambda|_\bullet))\wedge \omega^{n-1}}{\rank(\bullet)}$ denotes the slope of sheaf with respect to $\omega$, for $S(\bullet)$ being the singular locus of the sheaf. It is \emph{analytically polystable} if it decomposes as an orthogonal direct sum of analytically stable $\lambda$-flat bundles with the same slope.
  
\end{enumerate}
\end{definition}

 By the wonderful work of Simpson and Mochizuki, we have the following theorem.
 
\begin{theorem}[\textbf{Simpson--Mochizuki Correspondence}, Kobayashi--Hitchin version, \cite{CS1,TM2}]\label{m}
Fix  $\lambda\in \mathbb{C}$.
\begin{enumerate}
                 \item Let $X$ be a complex projective manifold   with a fixed ample line bundle $L$. A $\lambda$-flat bundle $(E,D^\lambda)$ over $X$ is $\mu_L$-polystable  with vanishing Chern classes if and only if there is a pluri-harmonic metric $h$ on $(E,\mathbb{D}^\lambda)$.
                 \item Let $(X,\omega)$ be a compact K\"{a}hler manifold, $(E,\mathbb{D}^\lambda, h_0)$ be an analytically stable $\lambda$-flat bundle. Then there exists a unique Hermitian metric $h$ such that $\det(h)=\det(h_0)$ and  the Hermitian--Einstein condition $ \Lambda_\omega G(h,\mathbb{D}^\lambda)^\bot=0$ holds, where $G(h,\mathbb{D}^\lambda)^\bot$ denotes the trace-free part of $G(h,\mathbb{D}^\lambda)$.
                     \item (Uniqueness of pluri-harmonic metric) Let $h_i$ ($i=1,2$) be the pluri-harmonic metric on the $\lambda$-flat bundle $(E,\mathbb{D}^\lambda)$, then
\begin{itemize}
      \item we have the decomposition of $\lambda$-flat bundles $(E,\mathbb{D}^\lambda)=\bigoplus(E_a, \mathbb{D}^\lambda_a)$ which is orthogonal with respect to both of $h_i$ ($i=1,2$),
       \item the restrictions $h_{i,a}$ of $h_i$ to $E_a$ satisfy $h_{1,a}=c_ah_{2,a}$  for positive constants $c_a$.
\end{itemize}
\end{enumerate}

\end{theorem}

\begin{remark}
This correspondence still holds for the case of stable parabolic logarithmic $\lambda$-flat bundles over a projective variety with a simple normal crossing divisor by imposing a compatibility condition of pluri-harmonic metric with the parabolic structure (for details see \cite{CS2,TM1,TM2}).
\end{remark}

As a direct application of the above theorem, we have the following correspondence as the interpolation of the usual Corlette--Simpson correspondence \cite{Cor,CS1,CS4} and the Riemann--Hilbert correspondence.

\begin{corollary}[\textbf{Simpson--Mochizuki Correspondence}, Categorical version, {\cite[Corollary 5.18]{TM2}}]\label{cor3.4}
Let $X$ be a complex projective manifold. Then for  any $\lambda\in \mathbb{C}$, there is an equivalence   between the category of  $\mu_L$-polystable $\lambda$-flat bundles with vanishing Chern classes and the category of semisimple representations of the fundamental group $\pi_1(X)$ into $\mathrm{GL}(r,\mathbb{C})$. This equivalence preserves tensor products, direct sums and duals.
\end{corollary}

\begin{proof}
For the case of $\lambda=0$, we have  the usual Simpson correspondence. So we assume $\lambda\neq 0$.
Let $(E,D^\lambda)$ be a $\mu_L$-polystable $\lambda$-flat bundle (with trivial characteristic numbers), then there is a pluri-harmonic metric $h$ on $E$. Therefore, we get
\begin{align*}
0=(\mathbb{D}^\lambda)^2&=(\lambda\partial_h+\bar{\partial}_h+\theta_h+\lambda\theta_h^\dagger)^2\\
&=\lambda(R(h)+[\theta_h,\theta_h^\dagger]+\partial_h\theta_h+\bar{\partial}_h\theta_h^\dagger),
\end{align*}
where $R(h)=(\mathcal{D}_h)^2$ is the curvature of the unitary connection $\mathcal{D}_h$, hence by Proposition \ref{kk}
$$R(h)+[\theta_h,\theta_h^\dagger]=\widetilde{\partial_h}\theta_h=\widetilde{\bar{\partial}_h}\theta_h^\dagger=0,$$
which  implies $((E,\bar{\partial}_h),\theta_h,h)$ is  a harmonic Higgs bundle associated with  a semi-simple representation $\rho: \pi_1(X)\to \mathrm{GL}(r,\mathbb{C})$ by the Hitchin--Simpson correspondence.

Conversely, if we have a semi-simple representation $\rho: \pi_1(X)\to \mathrm{GL}(r,\mathbb{C})$, then we have a  Higgs bundle $((E,\bar{\partial}_E),\theta,h)$ with the pluri-harmonic metric $h$, which gives rise to  a flat $\lambda$-connection $\mathbb{D}^\lambda=d_E'+d_E''$ with
$$
d_E'=\lambda\partial_{E,h}+\theta,\ \
d_E''=\bar{\partial}_E+\lambda\theta_h^\dagger,
$$
where $\partial_{E,h}$ is a (1,0)-type operator such that $\partial_{E,h}+\bar{\partial}_E$ is a unitary connection with respect to $h$, and $\theta_h^\dagger$ is the adjoint of $\theta$ with respect to $h$. Clearly, $h$ is also a pluri-harmonic metric for the $\lambda$-flat bundle $(E,\mathbb{D}^\lambda)$, hence it is polystable with trivial characteristic numbers.

Since pluri-harmonic metrics preserve tensor products, direct sums and duals, the equivalence described as above also preserves them.
\end{proof}

\begin{corollary}
If $X$ is a Riemann surface and $(E,\mathbb{D}^\lambda)$ is a stable $\lambda$-flat bundle over $X$ of rank $r\geq2$ and with vanishing the first Chern class, then there is no non-trivial global $\mathbb{D}^\lambda$-flat section of $E$.
\end{corollary}

\begin{proof}
When $\lambda=0$, the claim follows from \cite[Theorem 3.1]{Car}. Assume $\lambda\neq 0$. Let $h$ be  the pluri-harmonic on the stable $\lambda$-flat bundle $(E,\mathbb{D}^\lambda)$, and $s$ be the non-trivial global $\mathbb{D}^\lambda$-flat section, then the norm
$|s|_h^2$ is a sub-harmonic function by Proposition  \ref{w}. Since $X$ is compact,  $|s|_h^2$ is a nonzero constant, hence the section $s$ generates a trivial  line subbundle of $(E,\mathbb{D}^\lambda)$, which contradicts to the stability of $(E,\mathbb{D}^\lambda)$.
\end{proof}

\subsection{Moduli version}

Let $X$ be a complex projective manifold. Fixing $\lambda\in \mathbb{C}$, denote by  $\mathbf{M}^\lambda_{\mathrm{Hod}}(X,r)$ the moduli stack of rank $r$ $\lambda$-flat bundles with vanishing Chern classes over $X$, and by  $\mathbb{M}^\lambda_{\mathrm{Hod}}(X,r)$ the coarse moduli space for the semistable stratum of this stack, which is a quasi-projective variety and
 parameterizes the isomorphism classes of  polystable $\lambda$-flat bundles. Let $M^\lambda_{\mathrm{Hod}}(X,r)$ be the  smooth locus of $\mathbb{M}^\lambda_{\mathrm{Hod}}(X,r)$, which is  a Zariski dense open subset  and parameterizes the isomorphism classes of  stable $\lambda$-flat bundles. In particular, $\mathbb{M}^1_{\mathrm{Hod}}(X,r)=\mathbb{M}_{\mathrm{dR}}(X,r)$ and $\mathbb{M}^0_{\mathrm{Hod}}(X,r)=\mathbb{M}_{\mathrm{Dol}}(X,r)$.
 
Picking  a base point $x\in X$, we have the \emph{representation space} $\mathbb{R}_{\mathrm{Hod}}^\lambda(X,x,r)$, which is the fine moduli space of semistable  $\lambda$-flat bundles provided with a frame for the fiber over $x$, in particular, $\mathbb{R}^1_{\mathrm{Hod}}(X,r)=\mathbb{R}_{\mathrm{dR}}(X,r)$, and $\mathbb{R}^0_{\mathrm{Hod}}(X,r)=\mathbb{R}_{\mathrm{Dol}}(X,r)$. The group $\mathrm{GL}(r,\mathbb{C})$ acts on $\mathbb{R}_{\mathrm{Hod}}^\lambda(X,x,r)$, and $\mathbb{M}^\lambda_{\mathrm{Hod}}(X,r)=\mathbb{R}_{\mathrm{Hod}}^\lambda(X,x,r)/\!\!/\mathrm{GL}(r,\mathbb{C})$ as the universal categorical quotient. We also consider the subset $R_{\mathrm{Hod}}^\lambda(X,x,r)\subset \mathbb{R}_{\mathrm{Hod}}^\lambda(X,x,r)$ that consists of those points which admits a pluri-harmonic metric compatible with the frame at $x$. Such condition fixes the metric uniquely. The group $U(r)$ acts on $R_{\mathrm{Hod}}^\lambda(X,x,r)$, and $\mathbb{M}^\lambda_{\mathrm{Hod}}(X,r)={R}_{\mathrm{Hod}}^\lambda(X,x,r)/U(r)$ as the topological quotient.

Let $\mathbb{N}_{\mathrm{Hod}}^\lambda(X,r)$ be the Zariski  dense open subset of $\mathbb{M}^\lambda_{\mathrm{Hod}}(X,r)$ that parameterizes $\lambda$-flat bundles  such that the underlying vector bundles are semistable, which is an affine bundle over the coarse moduli space $\mathbb{B}(X,r)$ of semistable vector bundles of rank $r$ with  vanishing  Chern classes over $X$.

  \begin{proposition}
   Suppose $r\geq 2$.
  \begin{enumerate}
    \item Let $X$ be a  Riemann surface of genus $g\geq 2$. One defines $ \mathring{\mathbb{M}}_{\mathrm{Hod}}^\lambda(X,r)=\mathbb{M}^\lambda_{\mathrm{Hod}}(X,r)\backslash M^\lambda_{\mathrm{Hod}}(X,r)$. If both $M^\lambda_{\mathrm{Hod}}(X,r)$ and $ \mathring{\mathbb{M}}_{\mathrm{Hod}}^\lambda(X,r)$ are nonempty, then for the codimension of $\mathring{\mathbb{M}}_{\mathrm{Hod}}^\lambda(X,r)$ in $\mathbb{M}^\lambda_{\mathrm{Hod}}(X,r)$ we have 
    $$
    \mathrm{codim}_\mathbb{C}\mathring{\mathbb{M}}_{\mathrm{Hod}}^\lambda(X,r)\geq 2.
    $$
    \item Let $X$ be a  Riemann surface of genus $g\geq 3$. One defines $\widetilde{\mathbb{M}}_{\mathrm{Hod}}^\lambda(X,r)=\mathbb{M}^\lambda_{\mathrm{Hod}}(X,r)\backslash \mathbb{N}^\lambda_{\mathrm{Hod}}(X,r)$. Then for the codimension of $\widetilde{\mathbb{M}}_{\mathrm{Hod}}^\lambda(X,r)$ in $\mathbb{M}^\lambda_{\mathrm{Hod}}(X,r)$ we have 
    $$
    \mathrm{codim}_\mathbb{C}\widetilde{\mathbb{M}}_{\mathrm{Hod}}^\lambda(X,r)\geq 2.
    $$
  \end{enumerate}

  \end{proposition}

  \begin{proof} (1) For any partition $\overrightarrow{r}=(r_1,\cdots, r_k)\in \mathbb{Z}_{+}^{\oplus k}$ with $\sum_{i=1}^kr_i=r$ and $1<k\leq r$, we introduce a map
  $$
  \delta_{\overrightarrow{r}}:M^\lambda_{\mathrm{Hod}}(X,\overrightarrow{r}):=M^\lambda_{\mathrm{Hod}}(X,r_1)\times\cdots\times M^\lambda_{\mathrm{Hod}}(X,r_k)\rightarrow \mathbb{M}^\lambda_{\mathrm{Hod}}(X,r)
  $$
  by $((E_1,\theta_1),\cdots,(E_k,\theta_k))\mapsto (\bigoplus_{i=1}^kE_i,\bigoplus_{i=1}^k\theta_k)$. Since $\delta_{\overrightarrow{r}}$ is injective, we have
  $$
  \dim_\mathbb{C}\mathring{\mathbb{M}}_{\mathrm{Hod}}^\lambda(X,r)=\dim_\mathbb{C}\bigcup_{\{\overrightarrow{r}\}}\mathrm{Im}(\delta_{\overrightarrow{r}})=\max_{\{\overrightarrow{r}\}}\{\dim_\mathbb{C}M^\lambda_{\mathrm{Hod}}(X,\overrightarrow{r})\}.
  $$
  Hitchin and Simpson calculated the dimension of moduli space \cite{Hit,CS4,CS5} 
  $$
  \dim_\mathbb{C}M^\lambda_{\mathrm{Hod}}(X,r_i)=\dim_\mathbb{C}M_{\mathrm{Dol}}(X,r_i)=2r_i^2(g-1)+2,$$  then one can easily show that $$\max_{\{\overrightarrow{r}\}}\{\dim_\mathbb{C}M^\lambda_{\mathrm{Hod}}(X,\overrightarrow{r})\}=2(g-1)((r-1)^2+1)+4,
  $$
  which means that $\mathrm{codim}_\mathbb{C}\mathring{\mathbb{M}}_{\mathrm{Hod}}^\lambda(X,r)=4(g-1)(r-1)-2\geq 2$.

  (2) Let $N_{\mathrm{Hod}}^\lambda(X,r)=\mathbb{N}_{\mathrm{Hod}}^\lambda(X,r)\bigcap M_{\mathrm{Hod}}^\lambda(X,r)$,
   $\mathring{\mathbb{N}}_{\mathrm{Hod}}^\lambda(X,r)=\mathbb{N}_{\mathrm{Hod}}^\lambda(X,r)\backslash N_{\mathrm{Hod}}^\lambda(X,r)$, and $\widetilde{M}_{\mathrm{Hod}}^\lambda(X,r)=M^\lambda_{\mathrm{Hod}}(X,r)\backslash N^\lambda_{\mathrm{Hod}}(X,r)$.
   The same  argument as (1) shows that
   $$
   \mathrm{codim}_\mathbb{C}\mathring{\mathbb{N}}_{\mathrm{Hod}}^\lambda(X,r)=2(g-1)(r-1)-1\geq 3.
   $$
   Therefore, it suffices to prove $\mathrm{codim}_\mathbb{C}\widetilde{M}_{\mathrm{Hod}}^\lambda(X,r)\geq 2$.

   For a filtration $E=F^0\supset F^1\supset\cdots \supset F^{k-1}\supset F^k=0$ of subbundles of a given vector bundle $E$, the pair $(\overrightarrow{r},\overrightarrow{d})$ is called the \emph{type} of this filtration, where 
\begin{align*}
\overrightarrow{r}&=(\rank (F^{k-1}), \rank(F^{k-2}/F^{k-1}),\cdots,\rank(E/F^1)),\\
\overrightarrow{d}&=(\deg (F^{k-1}), \deg(F^{k-2}/F^{k-1}),\cdots,\deg(E/F^1)).
\end{align*}   
The moduli space $M_{\mathrm{Hod}}^\lambda(X,r)$ admits  a Harder--Narasimhan stratification
  $$
  M_{\mathrm{Hod}}^\lambda(X,r)=\coprod_{(\overrightarrow{r},\overrightarrow{d})}H_{(\overrightarrow{r},\overrightarrow{d})}(X,r),
  $$
  where the locally closed subset $H_{(\overrightarrow{r},\overrightarrow{d})}(X,r)$ of $M_{\mathrm{Hod}}^\lambda(X,r)$ parameterizes stable $\lambda$-flat bundles  such that the underlying vector bundles having Harder--Narasimhan  type $(\overrightarrow{r},\overrightarrow{d})$. Due to the boundedness of moduli space \cite{CS5}, there are finitely many Harder--Narasimhan types occur in the disjoint union.  The forgetful map $f:H_{(\overrightarrow{r},\overrightarrow{d})}(X,r)\rightarrow B_{(\overrightarrow{r},\overrightarrow{d})}(X,r)$
  via $(E,D^\lambda)\mapsto E$ gives rise to  a fibration over the space $ B_{(\overrightarrow{r},\overrightarrow{d})}(X,r)$ of isomorphism classes
  of vector bundles with Harder--Narasimhan type $(\overrightarrow{r},\overrightarrow{d})$ with fibers as Zariski open dense subsets of an affine space of dimension $d_f$.  By Riemann--Roch formula, $d_f$ is given by
  $$
  d_f=\dim_\mathbb{C}H^1(X,\End(E))-\dim_\mathbb{C}H^0(X,\End(E))+1=r^2(g-1)+1.
  $$
  Obviously, we have  
  $$
  \widetilde{M}_{\mathrm{Hod}}^\lambda(X,r)=(\coprod_{\overrightarrow{r}\neq (r),\overrightarrow{d}\neq (0)}H_{(\overrightarrow{r},\overrightarrow{d})}(X,r))\coprod \mathring{H}(X,r),
  $$ 
  where $\mathring{H}(X,r)$ is a subset of $H_{\overrightarrow{r}=(r),\overrightarrow{d}=(0)}(X,r)$ consisting of $\lambda$-flat bundles such that the underlying vector bundle is semistable but not stable. Then,  by a result of \cite{Bho} which shows that the dimension of $ B_{(\overrightarrow{r},\overrightarrow{d})}(X,r)$ is at most $r^2(g-1)-(r-1)(g-2)$ if $\overrightarrow{r}\neq (r),\overrightarrow{d}\neq (0)$, we conclude that
  $$
  \mathrm{codim}_\mathbb{C}\coprod_{\overrightarrow{r}\neq (r),\overrightarrow{d}\neq (0)}H_{(\overrightarrow{r},\overrightarrow{d})}(X,r)\geq 1+(r-1)(g-2)\geq 2.
  $$
  And by a result of \cite{BM} which asserts that the dimension of $ \mathring{H}(X,r)$ is at most $(2r^2-r+1)(g-1)+2$, we have
  $$\mathrm{codim}_\mathbb{C}  \mathring{H}(X,r)\geq (r-1)(g-1)\geq 2.$$
  From the above two inequalities the final result follows.
  \end{proof}

The proof of the following theorem is after Simpson (\cite[Lemma 7.13]{CS6}, \cite[Lemma 8.1]{CS9}) essentially.

\begin{theorem}\label{d}
The natural quotient map
$q:R_{\mathrm{Hod}}^\lambda(X,x,r)\rightarrow \mathbb{M}_{\mathrm{Hod}}^\lambda(X,r)$ is proper.
\end{theorem}

\begin{proof}
The cases of $\lambda=0,1$ have been proved by Simpson (\cite[Corollary 7.12, Corollary 7.15]{CS6}). 

For the case of $\lambda\neq 0,1$, we consider a sequence $\{((E_i,d^{\prime\prime}_{E_i}),D^\lambda_i,\beta_i)\}$ lying inside the inverse image of a compact subset of $\mathbb{M}_{\mathrm{Hod}}^\lambda(X,r)$, where $\beta_i$ is a frame on $E_{ix}$, and let $h_i$ be the unique pluri-harmonic metric on $((E_i,d^{\prime\prime}_{E_i}),D^\lambda_i,\beta_i)$.  It suffices to show that the characteristic polynomial of the  corresponding
Higgs fields $\{\theta_{h_i}\}$ are uniformly bounded in $C^0$-norm. By the map $((E,d^{\prime\prime}),D^\lambda,\beta)\mapsto ((E,d^{\prime\prime}),\lambda^{-1}D^\lambda,\beta)$ and the Riemann--Hilbert correspondence, $\mathbb{M}_{\mathrm{Hod}}^\lambda(X,r)$ is complex analytically isomorphic to $\mathbb{M}_\mathrm{B}(X,r)$, the coarse  moduli space of   representations $\pi_1(X,x)\to\mathrm{GL}(r,\mathbb{C})$. Let $\rho_i$ be the  monodromy representation corresponding to $((E_i,d^{\prime\prime}_{E_i}),D^\lambda_i,\beta_i)$, then $\{\rho_i\}$ lie over a compact subset of $\mathbb{M}_\mathrm{B}(X,r)$, hence the norms  $\{|\rho_i(\gamma)|\}=\{\sqrt{\Tr(\rho_i(\gamma)\rho_i^\dagger(\gamma))}\}$ are uniformly bounded for any generator $\gamma$ of $\pi_1(X,x)$. Denote by $\rho^{(\infty)}(\gamma)$ the limit point of $\{\rho_i(\gamma)\}$. By virtue of Mochizuki correspondence, each $\rho_i$ produces another simple monodromy representation $\tilde \rho_i$ of $\pi_1(X,x)$ given by the flat bundle $((E,\bar\partial_{h_i}+\bar \theta_{h_i}),\partial_{h_i}+\theta_{h_i},\beta_i)$, then the norms  $\{|\tilde\rho_i(\gamma)|\}\}$ are also  uniformly bounded. Indeed, we consider a family of flat bundles $((E,\bar\partial_{h_i}+t^{-1}\theta^\dagger_{h_i}),\partial_{h_i}+t\theta_{h_i})$ parameterized by $t\in \mathbb{C}^*$, and the associated family of monodromy representations is denoted by $\rho_t^{(i)}$. It is clear that the map $t\mapsto|\rho_t^{(i)}(\gamma)|$ is continuous.  We have the bound $|\rho_i(\gamma)|\leq C$. If $|\tilde \rho_i(\gamma)|$ tends to infinity, then for any   constant $C_1>C$, there is a sequence $\{t_i\}$ which lie in a curve segment joining $\lambda^{-1}$ to 1 but not passing through 0 such that   $|\rho_{t_{i}}^{(i)}(\gamma)|=C_1$. By \cite[Theorem 1]{CS3}, the map $\rho\mapsto |\rho(\gamma)|$ from $\mathbb{M}_\mathrm{B}(X,r)$ to $\mathbb{R}$ is proper, thus we may assume
 $\{\rho_{t_{i}}^{(i)}\}$ has a limit point $\rho_\diamondsuit$, then $|\rho_\diamondsuit(\gamma)|=C_1$.  We can also  assume the sequence  $\{t_i\}$ has the limit point $t_\infty$, then $\rho_\diamondsuit(\gamma)=\rho_{t_\infty}^{(\infty)}(\gamma)$ due to the separatedness of moduli space, whose norm has a bound $C_2$. If one takes $C_1>C_2$, we will get a contradiction, which lead to the uniform bound of $\{|\tilde \rho_i(\gamma)|\}$.

 Consequently, by \cite[Corollary 6]{CS3}, the $L^2$-norms $\{||\theta_{h_i}||_{L^2}\}$ are uniformly bounded. Since the maximum norm of an eigenvalue of a holomorphic matrix is a subharmonic function, the eigenforms of $\theta_{h_i}$ are uniformly bounded in $C^0$. So far,  we prove the claim on the characteristic polynomial of
Higgs fields $\{\theta_{h_i}\}$. Therefore, \cite[Lemma 2.8]{CS4}, or \cite[Proposition 7.9]{CS6} implies that there is a harmonic bundle $((E,\bar\partial),\theta,h,\beta)$, a subsequence $\{i'\}$ and $C^\infty$-automorphisms $g_{i'}$ such that $g_{i^\prime}^*(h_{i^\prime})=h$ and $g_{i^\prime}^*(\bar\partial_{h_{i^\prime}})-\bar\partial$, $g_{i^\prime}^*(\theta_{h_{i^\prime}})-\theta$ converge to zero strongly in the operator norm for operators from $L_1^p$ to $L^p$ for $p>1$, and the frames $g_{i^\prime}^*(\beta_{i^\prime})$ converge to $\beta$. Since the $\lambda$-flat bundle can be treated as  certain $\Lambda$-module in the sense of Simpson \cite{CS8}, \cite[Theorem 5.12]{CS5} is  valid for this case, hence there is a subsequence $\{((E_{i'},d^{\prime\prime}_{E_{i^\prime}}),D^\lambda_{i^\prime},\beta_{i^\prime})\}$ converge to a point $((E, \bar\partial+\lambda\theta^\dagger_h),\lambda\partial_h+\theta,\beta)$ in $R_{\mathrm{Hod}}^\lambda(X,x,r)$.
\end{proof}

\begin{corollary}[\textbf{Simpson--Mochizuki Correspondence}, Moduli version]\label{homeo}
The Simpson--Mochizuki correspondence described in Corollary \ref{cor3.4} provides a homeomorphism of moduli spaces
 $$
 \mathbb{M}_{\mathrm{Hod}}^\lambda(X,r)\simeq \mathbb{M}_{\mathrm{Dol}}(X,r).
 $$
\end{corollary}

\begin{proof}
A key step has been  completed in the proof of the above theorem, the remaining arguments are  totally parallel to \cite[Theorem 7.18]{CS6}, so we omit them here.
\end{proof}

\section{Mochizuki Correspondence on Balanced Manifolds}\label{sec4}

In this section, we always assume $X$ is an $n$-dimensional compact balanced manifold, and $((E,\bar{\partial}_E),D^\lambda)$ is a stable $\lambda$-flat bundle of $\mathrm{rank}(E)\geq2$ over $X$ with fixed $\lambda\neq0$.  We will use the standard method of continuity to show the existence  of harmonic metric on $((E,\bar{\partial}_E),D^\lambda)$.

Let $h_0$ be a fixed Hermitian metric on $E$, then both $\bar\partial_E+\delta_{h_0}'$ and $\lambda^{-1}D^\lambda+\bar{\lambda}^{-1}\delta_{h_0}''$ are $h_0$-unitary connections, whose curvatures are given by, respectively,
\begin{equation}\label{p01}
\begin{aligned}
  R_1(h_0)&:=(\bar\partial_E+\delta_{h_0}')^2=-\frac{(1+|\lambda|^2)}{\lambda}\bigg(\widetilde{\bar\partial_{h_0}}\theta_{h_0}+\lambda[\theta_{h_0},\theta_{h_0}^\dagger]\bigg),\\
R_2(h_0)&:=(\lambda^{-1}D^\lambda+\bar{\lambda}^{-1}\delta_h'')^2=-\frac{1+|\lambda|^2}{|\lambda|^2}\bigg(\lambda\widetilde{\partial_{h_0}}\theta^\dagger_{h_0}+[\theta_{h_0},\theta_{h_0}^\dagger]\bigg).
\end{aligned}
\end{equation}
Let $S(E,h_0)\subset  C^\infty(\End(E))$ be the set that consists of $h_0$-self-adjoint endomorphisms of $E$ and $S^+(E,h_0)$ be the subset of $S(E,h_0)$ that consists of positive-definite  endomorphisms. Write $h=h_0\cdot s$ for some $s\in S^+(E,h_0)$, then
\begin{align*}
                    \delta^\prime_{h}=\delta^\prime_{h_0}+s^{-1}\widetilde{\delta^\prime_{h_0}} s, \ \ \delta^{\prime\prime}_{h}=\delta^{\prime\prime}_{h_0}+s^{-1}\widetilde{\delta^{\prime\prime}_{h_0}} s,
                  \end{align*}
and we have the corresponding curvatures
 \begin{align*}
  R_1(h)&:= R_1(h_0)+\widetilde{\bar\partial_E}(s^{-1}\widetilde{\delta^\prime_{h_0}} s),\\
R_2(h)&:= R_2(h_0)+|\lambda|^{-2}\widetilde{D^\lambda}(s^{-1}\widetilde{\delta^{\prime\prime}_{h_0}} s).
\end{align*}

 If $h$ is a harmonic metric on $((E,\bar\partial_E),{D}^\lambda)$, then $s$ must satisfy the following equation
$$
K(h_0)+\sqrt{-1}\Lambda_\omega\widetilde{\bar\partial_E}(s^{-1}\widetilde{\delta^\prime_{h_0}} s)-\sqrt{-1}\Lambda_\omega\widetilde{D^\lambda}(s^{-1}\widetilde{\delta^{\prime\prime}_{h_0}} s)=0,
$$
where $K(h_0)=\sqrt{-1}\Lambda_\omega(R_1(h_0)-|\lambda|^2R_2(h_0))$. We consider the following perturbed equation
\begin{align}\label{oz}
  \Gamma_\varepsilon(h_0,s(\varepsilon)):=K(h_0)+\sqrt{-1}\Lambda_\omega(\widetilde{\bar\partial_E}(s^{-1}(\varepsilon)\widetilde{\delta^\prime_{h_0}} s(\varepsilon))-\widetilde{D^\lambda}(s^{-1}(\varepsilon)\widetilde{\delta^{\prime\prime}_{h_0}} s(\varepsilon)))+\varepsilon\log s(\varepsilon)=0
\end{align}
for some real number $\varepsilon$. One defines the set
$$
J(h_0)=\{\varepsilon\in(0,1]:\textrm{ there exists }s(\varepsilon)\in S^+(E,h_0)\textrm{ such that }\Gamma_\varepsilon(h_0,s(\varepsilon))=0\}.
$$
Given a metric $h$ on $E$, let $h_0=h\cdot e^{K(h)}$ so that $e^{K(h)}\in S^+(E,h_0)$, then $\Gamma_1(h_0,e^{-K(h)})=0$. Namely, one can choose a Hermitian metric $h_0$ on $E$ such that $1\in J(h_0)$.
\begin{lemma}\label{bh}
$J(h_0)$ is a nonempty  open subset of $(0,1]$.
\end{lemma}
\begin{proof}For a rational number $q$, one introduces the adjoint action $\mathrm{Ad}^q_s$ on an operator $\mathcal{O}\in C^\infty(\End(E)\otimes\Lambda^\bullet(T^*X))$
as $\mathrm{Ad}^q_s\cdot\mathcal{O}:=s^q\mathcal{O}s^{-q}$, which defines a new operator $\mathrm{Ad}^q_s(\mathcal{O})=\mathrm{Ad}^q_s\cdot(\mathcal{O}\comp \mathrm{Ad}^{-q}_s):\End(E)\rightarrow\End(E)\otimes\Lambda^\bullet(T^*X)$. Then we introduce the following notations
\begin{align*}
 P_1(q,s)=\mathrm{Ad}^q_s(\widetilde{\bar\partial_E})\comp \mathrm{Ad}^{-q}_s(\widetilde{\delta^\prime_{h_0}}),&\quad  \bar P_1(q,s)=\mathrm{Ad}^q_s(\widetilde{\delta^\prime_{h_0}})\comp\mathrm{Ad}^{-q}_s(\widetilde{\bar\partial_E})\\
 P_2(q,s)=\mathrm{Ad}^q_s(\widetilde{D^\lambda})\comp\mathrm{Ad}^{-q}_s(\widetilde{\delta^{\prime\prime}_{h_0}}), &\quad
 \bar P_2(q,s)=\mathrm{Ad}^q_s(\widetilde{\delta^{\prime\prime}_{h_0}})\comp \mathrm{Ad}^{-q}_s(\widetilde{D^\lambda}).
\end{align*}
In particular, we denote $ P_1^s=P_1(\frac{1}{2},s), P_2^s=P_2(\frac{1}{2},s), \bar P_1^s=\bar P_1(\frac{1}{2},s), \bar P_2^s=\bar P_2(\frac{1}{2},s)$.

We calculate
\begin{align*}
\frac{d}{dt}\bigg|_{t=0}\widetilde{\bar\partial_E}((s+t\eta)^{-1}\widetilde{\delta^\prime_{h_0}}(s+t\eta))
&= -\widetilde{\bar\partial_E}(s^{-1}\eta s\widetilde{\delta^\prime_{h_0}}s)+\widetilde{\bar\partial_E}(s^{-1}\widetilde{\delta^\prime_{h_0}}\eta)\\
&=\widetilde{\bar\partial_E}(s^{-1}\comp \widetilde{\delta^\prime_{h_0}}(\eta s^{-1})\comp s)\\
&=\ \mathrm{Ad}^{-\frac{1}{2}}_s(P_1^s(\eta^s)),
\end{align*}
where $\eta^{\underline{s}}=s^{-\frac{1}{2}}\eta s^{-\frac{1}{2}}$. Similarly, we have
\begin{align*}
\frac{d}{dt}\bigg|_{t=0}\widetilde{D^\lambda}((s+t\eta)^{-1}\widetilde{\delta^{\prime\prime}_{h_0}}(s+t\eta))
=\mathrm{Ad}^{-\frac{1}{2}}_s(P_2^s(\eta^{\underline{s}})).
\end{align*}
Therefore, the linearization of the equation \eqref{oz} at $(\varepsilon,s(\varepsilon))$ reads
\begin{align}\label{ppp}
  L_{\varepsilon,s}(\eta):=&\ \frac{d}{dt}|_{t=0}\Gamma_\varepsilon(h_0,s+t\eta)\nonumber\\
  =&\ \sqrt{-1}\Lambda_\omega(\mathrm{Ad}^{-\frac{1}{2}}_s(P_1^s(\eta^{\underline{s}}))-\mathrm{Ad}^{-\frac{1}{2}}_s(P_2^s(\eta^{\underline{s}})))+\varepsilon s^{-1}\eta
\end{align}

Since  the connections $$d_1^s:=\mathrm{Ad}^{\frac{1}{2}}_s({\bar\partial_E})+\mathrm{Ad}^{-\frac{1}{2}}_s({\delta^\prime_{h_0}}),\
d_2^s:=\lambda^{-1}\mathrm{Ad}^{\frac{1}{2}}_s(D^\lambda)+\bar\lambda^{-1}\mathrm{Ad}^{-\frac{1}{2}}_s({\delta^{\prime\prime}_{h_0}})$$
are also $h_0$-unitary,  we have
\begin{align}\label{mkr}
  \Delta_{\bar\partial}|\eta^{\underline{s}}|^2_{h_0}&=\sqrt{-1}\Lambda_\omega\bar\partial\partial|\eta^{\underline{s}}|^2_{h_0}=-\sqrt{-1}\Lambda_\omega\partial\bar\partial|\eta^{\underline{s}}|^2_{h_0}\nonumber\\&=h_0(\sqrt{-1}\Lambda_\omega P_1^s(\eta^{\underline{s}}),\eta^{\underline{s}})+h_0(\eta^{\underline{s}}, -\sqrt{-1}\Lambda_\omega \bar P_1^s(\eta^{\underline{s}}))-|\widetilde{d_1^s}\eta^{\underline{s}}|^2_{h_0,\omega}\nonumber\\
 &=|\lambda|^{-2}h_0(-\sqrt{-1}\Lambda_\omega P_2^s(\eta^{\underline{s}}),\eta^{\underline{s}})+|\lambda|^{-2}h_0(\eta^{\underline{s}}, \sqrt{-1}\Lambda_\omega \bar P_2^s(\eta^{\underline{s}}))-|\widetilde{d_2^s}\eta^{\underline{s}}|^2_{h_0,\omega},
\end{align}
where $ \Delta_{\bar\partial}=\sqrt{-1}\Lambda_\omega\bar\partial\partial$ is the Lapacian on $C^\infty(X)$ of $\bar\partial$ with respect to $\omega$.
Note that $(P_1^s(\eta^{\underline{s}}))^\dagger_{h_0}=\bar P_1^s(\eta^{\underline{s}}), (P_2^s(\eta^{\underline{s}}))^\dagger_{h_0}=\bar P_2^s(\eta^{\underline{s}}) $.

If $ L_{\varepsilon,s}(\eta)=0$, then combining \eqref{ppp} and \eqref{mkr} together leads to
\begin{align*}
(1+|\lambda|^2)\Delta_{\bar\partial}|\eta^{\underline{s}}|^2_{h_0}+|\widetilde{d_1^s}\eta^{\underline{s}}|^2_{h_0,\omega}+|\lambda|^2|\widetilde{d_2^s}\eta^{\underline{s}}|^2_{h_0,\omega}=-2\varepsilon|\eta^{\underline{s}}|^2_{h_0}\leq0.
\end{align*}
Hence, the maximum principle implies $|\eta^{\underline{s}}|=0$, i.e. $\eta=0$, which means the linear second order elliptic differential operator $  L_{\varepsilon,s}$ on $S(E,h_0)$ is injective. Moreover, since the index of $  L_{\varepsilon,s}$ is zero, it is also surjective. If  for some $\varepsilon_0\in (0,1]$ there exists $s_{\varepsilon_0}$ such that $\Gamma_{\varepsilon_0}(h_0,s_{\varepsilon_0})=0$, then by implicit function theorem on Banach spaces and elliptic regularity, there is a $S^+(E,h_0)$-valued smooth function $s(\varepsilon)$  over a small neighborhood $U\subset (0,1]$ of $\varepsilon_0$ with $s(\varepsilon_0)=s_{\varepsilon_0}$ such that $\Gamma_{\varepsilon}(h_0,s(\varepsilon))=0$ holds  for any $\varepsilon\in U$. The lemma follows. 
\end{proof}

\begin{lemma}\label{mmb}
Assume for any $\varepsilon\in (\epsilon,1]$ with $\epsilon>0$ the equation $\Gamma_\varepsilon(h_0,s(\varepsilon))=0$ admits a solution $s(\varepsilon)\in S^+(E,h_0)$. Denote  $\chi(\varepsilon)=\frac{ds(\varepsilon)}{d\varepsilon}$ and $\mathfrak{m}_\varepsilon=\max_X|\chi(\varepsilon)|_{h_0}$, then there exists a positive constant $C(m_\varepsilon)$ such that
$$
\max_X|\log s(\varepsilon)|_{h_0}\leq C(m_\varepsilon)
$$ 
for any $\varepsilon\in (\epsilon,1]$, where  the expression $C(m_\varepsilon)$ means  this constant depends  on  $m_\varepsilon$ (and other fixed data independent of $\varepsilon$).
\end{lemma}

\begin{proof}
The equation $\Gamma_\varepsilon(h_0,s(\varepsilon))=0$ is equivalent to $\widehat{\Gamma}_\varepsilon(h_0,s(\varepsilon)):=s(\varepsilon)\Gamma_\varepsilon(h_0,s(\varepsilon))=0$, then we have
\begin{align*}
 0&=\frac{d\widehat{\Gamma}_\varepsilon(h_0,s(\varepsilon))}{d\varepsilon}\\
 &=s(\varepsilon)(\mathrm{Ad}^{-\frac{1}{2}}_s(P_1^s((\chi(\varepsilon))^{\underline{s}}))-\mathrm{Ad}^{-\frac{1}{2}}_s(P_2^s((\chi(\varepsilon))^{\underline{s}}))+\log s(\varepsilon)+\varepsilon\Phi(\varepsilon)),
\end{align*}
where $\Phi(\varepsilon)=\frac{d\log s(\varepsilon)}{d\varepsilon}$. It is known that \cite{LT}
$$
h_0(\mathrm{Ad}^{\frac{1}{2}}_s\cdot \Phi(\varepsilon),(\chi(\varepsilon))^{\underline{s}})\geq |(\chi(\varepsilon))^{\underline{s}}|^2_{h_0},
$$
combined with  \eqref{mkr} yields
\begin{align}\label{11s}
  &(1+|\lambda|^2)\Delta_{\bar\partial}|(\chi(\varepsilon))^{\underline{s}}|^2_{h_0}+|\widetilde{d_1^s}(\chi(\varepsilon))^{\underline{s}}|^2_{h_0,\omega}
  +|\lambda|^2|\widetilde{d_2^s}((\chi(\varepsilon))^{\underline{s}}|^2_{h_0,\omega}
  +2\varepsilon|(\chi(\varepsilon))^{\underline{s}}|^2_{h_0}\nonumber\\
  \leq&-2h_0(\log s(\varepsilon),(\chi(\varepsilon))^{\underline{s}})
  \leq2|\log s(\varepsilon)|_{h_0}|(\chi(\varepsilon))^{\underline{s}}|_{h_0}.
\end{align}

On the other hand, we have
\begin{align*}
 |\widetilde{d_1^s}(\chi(\varepsilon))^{\underline{s}}|^2_{h_0,\omega}&\geq |\mathrm{Ad}^{\frac{1}{2}}_s(\widetilde{{\bar\partial_E}})(\chi(\varepsilon))^{\underline{s}}|^2_{h_0,\omega}\\
 &=| \mathrm{Ad}^{\frac{1}{2}}_s\cdot (\widetilde{{\bar\partial_E}}\varphi(\varepsilon))|^2_{h_0,\omega}\geq e^{-C_0(m_\varepsilon)}|\widetilde{{\bar\partial_E}}\varphi(\varepsilon)|^2_{h_0,\omega},
\end{align*}
where $\varphi(\varepsilon)=\mathrm{Ad}^{-\frac{1}{2}}_s\cdot (\chi(\varepsilon))^{\underline{s}}$, and $C_0(m_\varepsilon)$ is a constant. Taking the integration over $X$ yields
$||\widetilde{d_1^s}(\chi(\varepsilon))^{\underline{s}}||^2_{L^2}\geq ||\widetilde{{\bar\partial_E}}\varphi(\varepsilon)||^2_{L^2}$, where $||\bullet||_{L^2}$ denotes the $L^2$-norm with respect to $h_0, \omega$. Similarly, we have
$||\widetilde{d_2^s}(\chi(\varepsilon))^{\underline{s}}||^2_{L^2}\geq e^{-C_0(m_\varepsilon)}|\lambda|^{-2}||\widetilde{D^\lambda}\varphi(\varepsilon)||^2_{L^2}$. Consequently,
\begin{align*}
 ||\widetilde{d_1^s}(\chi(\varepsilon))^{\underline{s}}||^2_{L^2}+|\lambda|^2||\widetilde{d_2^s}((\chi(\varepsilon))^{\underline{s}}||^2_{L^2}
 &\geq  e^{-C_0(m_\varepsilon)}||\widetilde{\mathbb{D}^\lambda}\varphi(\varepsilon)||^2_{L^2}\nonumber\\
&=e^{-C_0(m_\varepsilon)}\int_Xh_0(\widetilde{\Delta_{\mathbb{D}^\lambda}}\varphi(\varepsilon),\varphi(\varepsilon))\omega^n\nonumber\\
&\geq a_1e^{-C_0(m_\varepsilon)}||\varphi(\varepsilon)||^2_{L^2}
\end{align*}
where $a_1$ is the smallest eigenvalue of the Laplacian $\widetilde{\Delta^{h_0}_{\mathbb{D}^\lambda}}=\widetilde{\mathbb{D}^\lambda}^*_{h_0}\widetilde{\mathbb{D}^\lambda}$ on $C^\infty(\End(E))$ of $\mathbb{D}^\lambda$ with respect to $h_0, \omega$.
One claims $a_1>0$. If $a_1=0$, namely $\widetilde{D^\lambda}\varphi(\varepsilon)=\widetilde{\bar\partial_E}\varphi(\varepsilon)=0$, which implies $\varphi(\varepsilon)=c(\varepsilon)\mathrm{Id}$ for some constant $c(\varepsilon)$ since $((E,\bar\partial_E),D^\lambda)$ is a stable $\lambda$-flat bundle. However, due to \eqref{p01} and $\omega$ being balanced, we have $\int_X\Tr (\sqrt{-1}\Lambda_\omega R_1(h))\omega^n=\int_X\Tr (\sqrt{-1}\Lambda_\omega R_2(h))\omega^n=0$, which implies
$$
\frac{d}{d\varepsilon}\int_X\Tr(\log s(\varepsilon))\omega^n=\int_X\Tr((\chi(\varepsilon))^{\underline{s}})\omega^n=\int_X\Tr(\varphi(\varepsilon))\omega^n=0.
$$
This means that $a_1=0$ only happens for $\chi(\varepsilon)=0$, then there exists a positive constant $C_1(m_\varepsilon)$ such that
 \begin{align}\label{kll}
  ||\widetilde{d_1^s}(\chi(\varepsilon))^{\underline{s}}||^2_{L^2}+|\lambda|^2||\widetilde{d_2^s}((\chi(\varepsilon))^{\underline{s}}||^2_{L^2}\geq
 e^{-C_1(m_\varepsilon)} ||(\chi(\varepsilon))^{\underline{s}}||^2_{L^2}.
 \end{align}
 Combining \eqref{11s} and \eqref{kll} together immediately gives rise to
\begin{align}\label{mk}
 ||(\chi(\varepsilon))^{\underline{s}}||_{L^2}\leq C_2(m_\varepsilon)
\end{align}
 for some positive constant $C_2(m_\varepsilon)$.

 Again by \eqref{11s}, we have
 \begin{align*}
   (1+|\lambda|^2)\Delta_{\bar\partial}|(\chi(\varepsilon))^{\underline{s}}|^2_{h_0}\leq2|\log s(\varepsilon)|_{h_0}|(\chi(\varepsilon))^{\underline{s}}|_{h_0}\leq m_\varepsilon|(\chi(\varepsilon))^{\underline{s}}|^2_{h_0}+m_\varepsilon.
 \end{align*}
Then one can apply \cite[Lemma 3.3.2]{LT} and \eqref{mk} to obtain
\begin{align*}
 \max_X|\log s(\varepsilon)|_{h_0}&\leq C_3(m_\varepsilon) \max_X|(\chi(\varepsilon))^{\underline{s}}|^2_{h_0}\\
 &\leq C_4(m_\varepsilon)(||(\chi(\varepsilon))^{\underline{s}}||^2_{L^2}+m_\varepsilon)\leq C_5(m_\varepsilon)
\end{align*}
where $ C_3(m_\varepsilon),  C_4(m_\varepsilon), C_5(m_\varepsilon)$ are positive constants. 
\end{proof}

\begin{lemma}\label{mmmm}
With the same setting as in Lemma \ref{mmb} we have
\begin{enumerate}
  \item $m_{\varepsilon}\leq \frac{1}{\epsilon}\max_X|K(h_0)|_{h_0}$
  \item $m_{\varepsilon}\leq \frac{C}{1+|\lambda|^2}\Big(||\log s(\varepsilon)||_{L_2}+\max_X|K(h_0)|_{h_0}\Big)^2$, where $C$ is a positive constant independent of $\varepsilon$.
\end{enumerate}
\end{lemma}

\begin{proof}
From $\Gamma_\varepsilon(h_0,s(\varepsilon))=0$ it follows that
\begin{align*}
  &|K(h_0)|_{h_0}|\log s(\varepsilon)|_{h_0}\\
  \geq&\  h_0(\sqrt{-1}\Lambda_\omega(\widetilde{\bar\partial_E}(s^{-1}(\varepsilon)\widetilde{\delta^\prime_{h_0}} s(\varepsilon))-\widetilde{D^\lambda}(s^{-1}(\varepsilon)\widetilde{\delta^{\prime\prime}_{h_0}} s(\varepsilon))),\log s(\varepsilon))+\varepsilon|\log s(\varepsilon)|^2_{h_0}
\end{align*}
By the same approach as in the proof of \cite[Lemma 3.3.4]{LT} , one can show that
\begin{align*}
 h_0(\sqrt{-1}\Lambda_\omega(\widetilde{\bar\partial_E}(s^{-1}(\varepsilon)\widetilde{\delta^\prime_{h_0}} s(\varepsilon)),\log s(\varepsilon))&\geq \frac{1}{2}\sqrt{-1}\Lambda_\omega\bar\partial\partial|\log s(\varepsilon)|_{h_0}^2,\\
 -h_0(\sqrt{-1}\Lambda_\omega(\widetilde{D^\lambda}(s^{-1}(\varepsilon)\widetilde{\delta^{\prime\prime}_{h_0}} s(\varepsilon)),\log s(\varepsilon))&\geq \frac{|\lambda|^2}{2}\sqrt{-1}\Lambda_\omega\bar\partial\partial|\log s(\varepsilon)|_{h_0}^2.
\end{align*}
Therefore, we arrive at
\begin{align}\label{c1v}
 \frac{1+|\lambda|^2}{2}\Delta_{\bar\partial}|\log s(\varepsilon)|_{h_0}^2+\varepsilon|\log s(\varepsilon)|^2_{h_0}\leq |K(h_0)|_{h_0}|\log s(\varepsilon)|_{h_0},
\end{align}
which gives the first estimate  in the lemma. The above inequality \eqref{c1v} also implies
\begin{align*}
 \Delta_{\bar\partial}|\log s(\varepsilon)|_{h_0}^2\leq \frac{2}{1+|\lambda|^2}|K(h_0)|_{k_0}|\log s(\varepsilon)|_{h_0}\leq\frac{1}{1+|\lambda|^2}\Big(|\log s(\varepsilon)|^2_{h_0}+\max_X|K(h_0)|^2_{h_0}\Big).
\end{align*}
Again by \cite[Lemma 3.3.2]{LT}, we get the second estimate  in the lemma. 
\end{proof}

 \begin{lemma}\label{gh}
The setting is the same as in Lemma \ref{mmb}. For all $p>1$ and $\varepsilon\in(\epsilon,1]$, there exists  positive constants $C, C'$ indepent of $\varepsilon$ such that
\begin{enumerate}
  \item $\|\chi(\varepsilon)\|_{L_2^p}\leq C(1+\|s(\varepsilon)\|_{L_2^p})$,
  \item $\|s(\varepsilon)\|_{L_2^p}\leq C^\prime$.
\end{enumerate}
\end{lemma}
\begin{proof}(1) We define the Laplacians as follows \begin{align*}
 \Delta_1^{h_0}&:= (\bar\partial_E+\delta_{h_0}')^*_{h_0,\omega}\comp(\bar\partial_E+\delta_{h_0}')+(\bar\partial_E+\delta_{h_0}')\comp(\bar\partial_E+\delta_{h_0}')^*_{h_0,\omega},\\
\Delta_2^{h_0}&:= (\lambda^{-1}D^\lambda+\bar{\lambda}^{-1}\delta_{h_0}'')^*_{h_0,\omega}\comp(\lambda^{-1}D^\lambda+\bar{\lambda}^{-1}\delta_{h_0}'')+(\lambda^{-1}D^\lambda+\bar{\lambda}^{-1}\delta_{h_0}'')\comp(\lambda^{-1}D^\lambda+\bar{\lambda}^{-1}\delta_{h_0}'')^*_{h_0,\omega}.
\end{align*}
then for any $\Xi\in C^\infty(\End(E))$,  we have
\begin{align*}
\widetilde{\Delta_1^{h_0}}\Xi&=2\sqrt{-1}\Lambda_\omega\widetilde{\bar\partial_E}\widetilde{\delta_{h_0}'}\Xi-[\sqrt{-1}\Lambda_\omega R_1({h_0}),\Xi],\\
\widetilde{\Delta_2^{h_0}}\Xi&=-2|\lambda|^{-2}2\sqrt{-1}\Lambda_\omega \widetilde{D^\lambda}\widetilde{\delta_{h_0}''}\Xi+[\sqrt{-1}\Lambda_\omega R_2({h_0}),\Xi].
\end{align*}
The identity $\frac{d\Gamma_\varepsilon(h_0,s(\varepsilon))}{d\varepsilon}=0$ gives rise to
\begin{align}\label{lml}
\widetilde{\Delta_1^{h_0}}\chi(\varepsilon)+|\lambda|^2\widetilde{\Delta_2^{h_0}}\chi(\varepsilon)\nonumber
=& -[K(h_0),\chi(\varepsilon)])
+ 2\sqrt{-1}\Lambda_\omega[\chi(\varepsilon)\comp s^{-1}(\varepsilon)\comp\widetilde{D^\lambda}s(\varepsilon)\comp s^{-1}(\varepsilon)\comp\widetilde{\delta_{h_0}''}s(\varepsilon)\nonumber\\
 &-\widetilde{D^\lambda}\chi(\varepsilon)\comp s^{-1}(\varepsilon)\comp\widetilde{\delta_{h_0}''}s(\varepsilon)
+\widetilde{D^\lambda}s(\varepsilon)\comp s^{-1}(\varepsilon)\comp\chi(\varepsilon)\comp s^{-1}(\varepsilon)\comp\widetilde{\delta_{h_0}''}s(\varepsilon)\nonumber\\
&-\widetilde{D^\lambda}s(\varepsilon)\comp s^{-1}(\varepsilon)\comp\widetilde{\delta_{h_0}''}\chi(\varepsilon)
-\chi(\varepsilon)\comp s^{-1}(\varepsilon)\comp\widetilde{D^\lambda}\widetilde{\delta_{h_0}''}s(\varepsilon)\nonumber\\
&-
\chi(\varepsilon)\comp s^{-1}(\varepsilon)\comp\widetilde{\bar\partial_E}s(\varepsilon)\comp s^{-1}(\varepsilon)\comp\widetilde{\delta_{h_0}'}s(\varepsilon)+\widetilde{\bar\partial_E}\chi(\varepsilon)\comp s^{-1}(\varepsilon)\comp\widetilde{\delta_{h_0}'}s(\varepsilon)\nonumber\\
&-\widetilde{\bar\partial_E}s(\varepsilon)\comp s^{-1}(\varepsilon)\comp\chi(\varepsilon)\comp s^{-1}(\varepsilon)\comp\widetilde{\delta_{h_0}'}s(\varepsilon)+\widetilde{\bar\partial_E}s(\varepsilon)\comp s^{-1}(\varepsilon)\comp\widetilde{\delta_{h_0}'}\chi(\varepsilon)\nonumber\\
&+\chi(\varepsilon)\comp s^{-1}(\varepsilon)\comp\widetilde{\bar\partial_E}\widetilde{\delta_{h_0}'}s(\varepsilon)]
+\varepsilon s(\varepsilon)\log s(\varepsilon)+s(\varepsilon)\Phi(\varepsilon).
\end{align}
The $L^p$-norms of all  terms on the right hand side of \eqref{lml} can be estimated, for example
\begin{align*}
 ||[K(h_0),\chi(\varepsilon)]||_{L^p}&\leq C_6||\chi(\varepsilon)||_{L^p}\leq C_7(m_\varepsilon),\\
 \|\sqrt{-1}\Lambda_\omega(\chi(\varepsilon)\comp s^{-1}(\varepsilon)\comp\widetilde{D^\lambda}s(\varepsilon)\comp s^{-1}(\varepsilon)\comp\widetilde{\delta_{h_0}''}s(\varepsilon))\|_{L^p}
&\leq C_8(m_\varepsilon)\|\widetilde{D^\lambda}s(\varepsilon)\comp\widetilde{\delta_h''}s(\varepsilon)\|_{L^p}\\
&\leq C_{8}(m_\varepsilon)\|s(\varepsilon)\|_{L_1^{2p}}^2,\\
\|\sqrt{-1}\Lambda_\omega(\widetilde{D^\lambda}\chi(\varepsilon)\comp s^{-1}(\varepsilon)\comp\widetilde{\delta_{h_0}''}s(\varepsilon))\Big\|_{L^p}
&\leq C_{9}(m_\varepsilon)\|\widetilde{D^\lambda}\chi(\varepsilon)\comp\widetilde{\delta_h''}s(\varepsilon)\|_{L^p}\\
&\leq C_{9}(m_\varepsilon)\|\chi(\varepsilon)\|_{L_1^{2p}}\cdot\|s(\varepsilon)\|_{L_1^{2p}},\\
||\sqrt{-1}\Lambda_\omega(\chi(\varepsilon)\comp s^{-1}(\varepsilon)\comp\widetilde{D^\lambda}\widetilde{\delta_{h_0}''}s(\varepsilon))||_{L^p}&\leq C_{10}(m_\varepsilon)||s(\varepsilon)||_{L_2^p},\\
||s(\varepsilon)\Phi(\varepsilon)||_{L^p}&\leq C_{11}(m_\varepsilon)||\chi(\varepsilon)||_{L^p}\leq C_{12}(m_\varepsilon),
\end{align*}
where we frequently use Lemma \ref{mmb} and H\"{o}lder inequality.

Since the operator
$
\Delta^{h_0}_1+|\lambda|^2\Delta^{h_0}_2+\mathrm{id}: L_2^p\longrightarrow L^p
$
is self-adjoint and has strictly positive spectrum, there exists a positive constant $C_{12}$ such that
\begin{align*}
\|\chi(\varepsilon)\|_{L_2^p}
\leq C_{12}(\|\chi(\varepsilon)\|_{L^p}+\|(\widetilde{\Delta^{h_0}_1}+|\lambda|^2\widetilde{\Delta^{h_0}_2})\chi(\varepsilon)\|_{L^p}).
\end{align*}
Then due to Lemma \ref{mmmm} (1), we get
\begin{align*}
\|\chi(\varepsilon)\|_{L_2^p}
&\leq C_{13}(1+\|s(\varepsilon)\|_{L_1^{2p}}^2+\|s(\varepsilon)\|_{L_2^p}+\|\chi(\varepsilon)\|_{L_1^{2p}}\|s(\varepsilon)\|_{L_1^{2p}})\\
&\leq C_{14}(1+\|s(\varepsilon)\|_{L_2^p}).
\end{align*}

(2) By the approach applied to the proof of \cite[Proposition 3.3.5 ii)]{LT}, we deduce the inequality
$$\|s(\varepsilon)\|_{L_2^p}\leq e^{C(1-t)}(1+\|s(1)\|_{L_2^p})$$ from (1).
This immediately implies (2).
\end{proof}

\begin{lemma}\
\begin{enumerate}
  \item $J=(0,1]$.
  \item If there is a positive constant $C$ such that $||s(\varepsilon)||_{L^2}\leq C$ for all $\varepsilon\in(0,1]$, then there exists a solution $s(0)$ of the equation $\Gamma_0(h_0,s(0))=0$.
\end{enumerate}
\end{lemma}

\begin{proof}
(1) Thanks to Lemma \ref{bh}, to show (1) it suffices to prove  $J$ is a closed subset of $(0,1]$. By Lemma \ref{gh} (2), $s(\varepsilon)$ is  uniformly bounded in $L_2^p(S^+(E,h_0))$ for  $\varepsilon\in(\epsilon,1]$, thus
there is a sequence $\{\varepsilon_i\in(\epsilon,1]\}_{i\in\mathbb{N}}$ converges to $\epsilon$ such that the sequence $\{s(\varepsilon_i)\}_{i\in\mathbb{N}}$ converges weakly to $s(\epsilon)\in L_2^p(S^+(E,h_0))$.
Since the Sobolev embedding $L_2^p\hookrightarrow L_1^p$ is compact, we may assume $\{s(\varepsilon_i)\}_{i\in\mathbb{N}}$ converges strongly to $s(\epsilon)$ in $L_1^p(S^+(E,h_0))$. Some rather standard arguments (cf. the proof of \cite[Proposition 3.3.6]{LT}) show that $s(\epsilon)$ is actually differentiable and satisfies $L_\epsilon(h_0,s(\epsilon))=0$.

(2) If $||s(\varepsilon)||_{L^2}\leq C$ for all $\varepsilon\in(0,1]$, then by Lemma \ref{mmmm} (2) and  and Lemma \ref{gh}
(2), $||s(\varepsilon)||_{L^p_2}$ is uniformly bounded on $(0,1]$, then same argument as in (1) shows (2). 
\end{proof}

\begin{lemma}\label{plo}
The setting is the same as in Lemma \ref{mmb}.  There is a positive constant $C$ independent of $\varepsilon$ such that
$$
\max_X|s(\varepsilon)|_{h_0}\leq C||s(\varepsilon)||_{L^1}.
$$
\end{lemma}

\begin{proof} 
Again by $\Gamma_\varepsilon(h_0,s(\varepsilon))=0$ we have
\begin{align*}
0=&\ h_0(K(h_0),s(\varepsilon))+\sqrt{-1}\Lambda_\omega \widetilde{\bar\partial} h_0(s^{-1}(\varepsilon)\widetilde{\delta^\prime_{h_0}} s(\varepsilon),s(\varepsilon))+\sqrt{-1}\Lambda_\omega h_0(s^{-1}(\varepsilon)\widetilde{\delta^\prime_{h_0}} s(\varepsilon),\widetilde{\delta^\prime_{h_0}} s(\varepsilon))\\
&-\sqrt{-1}\Lambda_\omega \lambda\partial h_0(s^{-1}(\varepsilon)\widetilde{\delta^{\prime\prime}_{h_0}} s(\varepsilon),s(\varepsilon))-
\sqrt{-1}\Lambda_\omega  h_0(s^{-1}(\varepsilon)\widetilde{\delta^{\prime\prime}_{h_0}} s(\varepsilon),\widetilde{\delta^{\prime\prime}_{h_0}}s(\varepsilon))+\varepsilon h_0(\log s(\varepsilon),s(\varepsilon)).
\end{align*}
Then by means of the identities
\begin{align*}
  \sqrt{-1}\Lambda_\omega \widetilde{\bar\partial} h_0(s^{-1}(\varepsilon)\widetilde{\delta^\prime_{h_0}} s(\varepsilon),s(\varepsilon))&=\Delta_{\bar\partial}\Tr s(\varepsilon),\\
  \sqrt{-1}\Lambda_\omega \lambda\partial h_0(s^{-1}(\varepsilon)\widetilde{\delta^{\prime\prime}_{h_0}} s(\varepsilon),s(\varepsilon))&=-|\lambda|^2\Delta_{\bar\partial}\Tr s(\varepsilon),
\end{align*}
and the inequalities
\begin{align*}
 \sqrt{-1}\Lambda_\omega h_0(s^{-1}(\varepsilon)\widetilde{\delta^\prime_{h_0}} s(\varepsilon),\widetilde{\delta^\prime_{h_0}} s(\varepsilon))&\geq |s^{-\frac{1}{2}}(\varepsilon)\widetilde{\delta^\prime_{h_0}} s(\varepsilon)|^2_{h_0},\\
 -
\sqrt{-1}\Lambda_\omega  h_0(s^{-1}(\varepsilon)\widetilde{\delta^{\prime\prime}_{h_0}} s(\varepsilon),\widetilde{\delta^{\prime\prime}_{h_0}}s(\varepsilon))&\geq |s^{-\frac{1}{2}}(\varepsilon)\widetilde{\delta^{\prime\prime}_{h_0}}s(\varepsilon)|^2_{h_0},
\end{align*}
we get
\begin{align*}
   \Delta_{\bar\partial}\Tr s(\varepsilon)&\leq -\frac{1 }{1+|\lambda|^2}(\varepsilon h_0(\log s(\varepsilon),s(\varepsilon))+h_0(K(h_0),s(\varepsilon)))\\
&\leq \frac{2 }{1+|\lambda|^2}\max_X|K(h_0)|_{h_0}  |s(\varepsilon)|_{h_0}\leq C_{15}\Tr s(\varepsilon)
\end{align*}
for some positive constant $C_{15}$, where we have applied Lemma \ref{mmmm} (1) for the third inequality. Note that $\max_X|s(\varepsilon)|_{h_0}\leq C_{16} \max_X\Tr s(\varepsilon)$ for some positive constant $C_{16}$, then applying \cite[Lemma 3.3.2]{LT} once again gives rise to the lemma.
\end{proof}

Assume $s(\varepsilon)$ is a solution to $\Gamma_\varepsilon(h_0,s(\varepsilon))=0$ for some $\varepsilon>0$. For $x\in X$, let $e(\varepsilon,x)$ be the largest eigenvalue of $\log s(\varepsilon)|_x$, then one defines $E(\varepsilon)=\max_Xe(\varepsilon,x)$, $\rho(\varepsilon)=e^{-E(\varepsilon)}$ and $S(\varepsilon)=\rho(\varepsilon)s(\varepsilon)$.

\begin{lemma}\label{jm}
The setting is the same as in Lemma \ref{mmb}. If $\lim\limits_{\varepsilon\rightarrow0}||s(\varepsilon)||_{L^2}=\infty$, then there exists a sequence $\{\varepsilon_i\rightarrow0\}$ such that $\rho(\varepsilon_i)\rightarrow 0$ and $S(\varepsilon_i)$ converges weakly in $L_1^2$-norm to $S_\infty\neq 0$.
\end{lemma}

\begin{proof}
By definition, $S(\varepsilon)\leq \mathrm{Id}_E$ and $\max_X(\rho(\varepsilon)|s(\varepsilon)|_{h_0})\geq 1$, then by Lemma \ref{plo}    we get
\begin{align*}
  1\leq \max_X(\rho(\varepsilon)|s(\varepsilon)|_{h_0})\leq C_{17}\rho(\varepsilon)||s(\varepsilon)||_{L^1}\leq C_{18}||S(\varepsilon)||_{L^2}\leq C_{19}
\end{align*} for some positive constants $C_{17}, C_{18}, C_{19}$. Firstly, from  $1\leq C_{18}||S(\varepsilon)||_{L^2}$ it follows that if $S(\varepsilon_i)$ converges to $S_\infty$ weakly in $L_1^2$-norm then  $S_\infty\neq 0$. Secondly, by $||S(\varepsilon)||_{L^2}\leq C_{19}$ we see that in order to obtain the $L_1^2$-bound for $S(\varepsilon_i)$ we only need to estimate $||(\widetilde{\bar\partial}+\widetilde{\delta^\prime_{h_0}})S(\varepsilon_i)||_{L^2}$. The same calculations as  in the proof of Lemma \ref{plo} show that \begin{align*}
                  ||(\widetilde{\bar\partial}+\widetilde{\delta^\prime_{h_0}})S(\varepsilon_i)||_{L^2}&\leq  2 ||\widetilde{\delta^\prime_{h_0}}S(\varepsilon)||_{L^2}  \leq    2(||S^{-\frac{1}{2}}(\varepsilon)\widetilde{\delta^{\prime}_{h_0}}S(\varepsilon)||_{L^2} +   ||S^{-\frac{1}{2}}(\varepsilon)\widetilde{\delta^{\prime\prime}_{h_0}}S(\varepsilon)||_{L^2})\\
                    &\leq 4\max_X|K(h_0)|_{h_0}\int_X  |S(\varepsilon)|_{h_0}\omega^n\leq C_{19}
                            \end{align*}
for some positive constant $C_{20}$, i.e. $S(\varepsilon)$ is  uniformly bounded in $L_1^2(S^+(E,h_0))$, which implies  $S(\varepsilon_i)$ converges weakly in $L_1^2$-norm.
\end{proof}

We have shown that there  is a sequence $\{\varepsilon_i\rightarrow 0\}$  such that  $S(\varepsilon_i)$ converges weakly to a nonzero $L_1^2$-endomorphism  $S_\infty$. Similarly, for $0< \varsigma\leq 1$, there is a sequence $\{\varepsilon_i\rightarrow0\}$ such that  $S^\varsigma(\varepsilon_i)$ converges weakly to a nonzero $L_1^2$-endomorphism  $S^\varsigma_\infty$, and there is a sequence $\{\varsigma_i\rightarrow0\}$ such that $S^{\varsigma_i}_\infty\rightarrow S^0_\infty$ weakly in $L_1^2$. Then one introduces an $L_1^2$-endomorphism $\Theta=\mathrm{Id}_E-S^0_\infty$.

\begin{lemma}\label{as}
$\Theta$ satisfies the following identities in $L^1$:
\begin{enumerate}
  \item $\Theta^2=\Theta=\Theta^\dagger_{h_0}$,
  \item $(\mathrm{Id}_E-\Theta)\comp\widetilde{\bar\partial}\Theta=(\mathrm{Id}_E-\Theta)\comp\widetilde{D^\lambda}\Theta=0$.
\end{enumerate}
Therefore, $\Theta$ defines a proper  $\lambda$-flat coherent subsheaf  $\mathcal{F}$ of $E$ with $0<\rank(\mathcal{F})<\rank(E)$.
\end{lemma}
\begin{proof}(1) is obvious. For (2), it suffices to show
\begin{align*}
  |\Theta\comp \widetilde{\delta^\prime_{h_0}}(\mathrm{Id}_E-\Theta)|_{h_0}=|\Theta\comp \widetilde{\delta''_{h_0}}(\mathrm{Id}_E-\Theta)|_{h_0}=0.
\end{align*}
Indeed, by the same arguments as in the proof of \cite[Proposition 3.4.6 iii)]{LT}, we have
\begin{align*}
 &||\Theta\comp \widetilde{\delta^\prime_{h_0}}(\mathrm{Id}_E-\Theta)||_{L^2}+||\Theta\comp \widetilde{\delta''_{h_0}}(\mathrm{Id}_E-\Theta)||_{L^2}\\
 =&\lim\limits_{\frac{\varsigma_k}{2}\geq\varsigma_j\rightarrow0}\lim\limits_{2\geq\varsigma_k\rightarrow0}\lim\limits_{\varepsilon_i\rightarrow0}
 (||(\mathrm{Id}_E-S^{\varsigma_j}(\varepsilon_i))\comp \widetilde{\delta^\prime_{h_0}}S^{\varsigma_k}(\varepsilon_i)||_{L^2}
+ ||(\mathrm{Id}_E-S^{\varsigma_j}(\varepsilon_i))\comp \widetilde{\delta''_{h_0}}S^{\varsigma_k}(\varepsilon_i)||_{L^2})\\
\leq&\lim\limits_{\frac{\varsigma_k}{2}\geq\varsigma_j\rightarrow0}\lim\limits_{2\geq\varsigma_k\rightarrow0}\lim\limits_{\varepsilon_i\rightarrow0}
(\frac{2\varsigma_j}{2\varsigma_j+\varsigma_k})^2(||S^{-\frac{\varsigma_k}{2}}(\varepsilon_i))\comp \widetilde{\delta^\prime_{h_0}}S^{\varsigma_k}(\varepsilon_i)||_{L^2}+||S^{-\frac{\varsigma_k}{2}}(\varepsilon_i))\comp \widetilde{\delta''_{h_0}}S^{\varsigma_k}(\varepsilon_i)||_{L^2})\\
\leq &\lim\limits_{\frac{\varsigma_k}{2}\geq\varsigma_j\rightarrow0}\lim\limits_{2\geq\varsigma_k\rightarrow0}
C_{21}(\frac{2\varsigma_j}{2\varsigma_j+\varsigma_k})^2\max_X|K(h_0)|_{h_0}\\
=&\ 0
\end{align*}
 for some positive constant $C_{21}$, which leads to (2). The existence of $\lambda$-flat coherent sheaf $\mathcal{F}$ defined via $\Theta$ is just an application of   a classical result due to  Uhlenbeck and Yau \cite{UY}. The nonvanishing of $S_\infty$ guarantees $\rank (\mathcal{F})<\rank (E)$, and the identity $\int_X\log(\det s(\varepsilon))\omega^n=0$ implies $0<\rank (\mathcal{F})$ (cf. the proof of \cite[Proposition 3.13 (3)]{HH2}). 
\end{proof}

\begin{lemma}
$\mathcal{F}$ is defined as in  Lemma \ref{as}, then $\deg(\mathcal{F})\geq 0$.
\end{lemma}

\begin{proof}
 There is an analytic subset $S\subset X$ with $\mathrm{codim}_XS\geq 2$ such that $\mathcal{F}|_{X\backslash S}$ is a holomorphic subbundle of $E|_{X\backslash S}$ and $\Theta|_{X\backslash S}$ is $C^\infty$ \cite{UY}. Since
\begin{align*}
\int_{X\backslash S}\Tr(\sqrt{-1}\Lambda_\omega R_1(h_0|_{\mathcal{F}})) \omega^n&=\frac{1}{1+|\lambda|^2} \int_{X\backslash S}\Tr(\sqrt{-1}\Lambda_\omega G(h_0|_{\mathcal{F}}, \mathbb{D}^\lambda|_{\mathcal{F}}))\omega^n\\
 &=\frac{1}{1+|\lambda|^2}\int_{X\backslash S}(\Tr(K(h_0)\comp\Theta)-|\widetilde{\delta'_{h_0}}\Theta|^2_{h_0}-|\widetilde{\delta''_{h_0}}\Theta|^2_{h_0})\omega^n,
\end{align*}
we only need to show
\begin{align}\label{my}
\int_{X\backslash S}\Tr(K(h_0)\comp\Theta)\omega^n\geq \int_{X\backslash S}|\widetilde{\delta'_{h_0}}\Theta|^2_{h_0}+|\widetilde{\delta''_{h_0}}\Theta|^2_{h_0})\omega^n.
\end{align}
By the identities   $\int_X\Tr(K(h_0))\omega^n=0$ and $\Theta=\lim\limits_{1\leq\varsigma_i\rightarrow0}\lim\limits_{\varepsilon_i\rightarrow0}(\mathrm{Id}_E-S^{\varsigma_i}(\sigma_i))$ (strongly in $L^2$), we have
\begin{align*}
 \int_{X\backslash S}\Tr(K(h_0)\comp\Theta)\omega^n=-\lim\limits_{1\leq\varsigma_i\rightarrow0}\lim\limits_{\varepsilon_i\rightarrow0}\int_{X}\Tr(K(h_0)\comp S^{\varsigma_i}(\varepsilon_i))\omega^n.
\end{align*}
A calculation of reuse shows that
\begin{align*}
 &\int_{X}\Tr(K(h_0)\comp S^{\varsigma_i}(\varepsilon_i))\omega^n\\
 =&
 -\int_{X}\sqrt{-1}\Lambda_\omega (h_0(S^{-1}(\varepsilon_i)\widetilde{\delta^\prime_{h_0}} S(\varepsilon_i),\widetilde{\delta^\prime_{h_0}} S^{\varsigma_i}(\varepsilon_i))-h_0(S^{-1}(\varepsilon_i)\widetilde{\delta''_{h_0}} S(\varepsilon_i),\widetilde{\delta''_{h_0}} S^{\varsigma_i}(\varepsilon_i)))\omega^n\\
 \leq &-||S^{-\frac{\varsigma_i}{2}}(\varepsilon_i)\widetilde{\delta^\prime_{h_0}} S^{\varsigma_i}(\varepsilon_i)||_{L^2}-||S^{-\frac{\varsigma_i}{2}}(\varepsilon_i)\widetilde{\delta''_{h_0}} S^{\varsigma_i}(\varepsilon_i)||_{L^2}\\
 \leq &-||\widetilde{\delta^\prime_{h_0}} (\mathrm{Id}_E-S^{\varsigma_i}(\varepsilon_i))||_{L^2}-||\widetilde{\delta''_{h_0}} (\mathrm{Id}_E-S^{\varsigma_i}(\varepsilon_i))||_{L^2},
\end{align*}
where we have also used the assumption that  $X$ being a balanced manifold and the inequality $\int_X\Tr(\log S(\varepsilon_i)\comp S^{\varsigma_i}(\varepsilon_i))\omega^n\geq0$. This immediately gives rise to the desired inequality \eqref{my}. 
\end{proof}

\begin{lemma}
If both $h_1, h_2$ are harmonic metrics on $((E,\bar\partial_E),{D}^\lambda)$, then there is a positive constant $C$ such that $h_2=Ch_1$.
\end{lemma}

\begin{proof}
Write $h_2=h_1 s$ with $s\in S^+(E,h_1) $, then from the proof of Lemma \ref{plo} we see that
\begin{align*}
  (1+|\lambda|^2)\Delta_{\bar\partial}\Tr s+|s^{-\frac{1}{2}}(\varepsilon)\widetilde{\delta^\prime_{h_1}} s|^2_{h_1}+|s^{-\frac{1}{2}}(\varepsilon)\widetilde{\delta''_{h_1}} s|^2_{h_1}\leq 0,
\end{align*}
which implies $\widetilde{\bar\partial_E}s=\widetilde{D^\lambda}s=0$. Since $((E,\bar\partial_E),{D}^\lambda)$ is a stable $\lambda$-flat bundle, this makes $s=C\cdot\mathrm{Id}_E$ for some positive constant $C$. 
\end{proof}

In conclusion, we achieve the following theorem:

\begin{theorem} \label{pluri-har}
Let $X$ be a balanced manifold, and  $((E,\bar\partial_E),{D}^\lambda)$ be a stable $\lambda$-flat bundle over $X$ ($\lambda\neq 0$), then there is a unique harmonic metric on $((E,\bar\partial_E),{D}^\lambda)$ up to constant scalars.

\end{theorem}

\section{Dynamical Systems on Dolbeault Moduli Spaces}\label{sec5}

\subsection{Construction}

In this section, $X$ is assumed to be a complex projective manifold.
 For any $t\in\mathbb{C}^*$, the $\mathbb{C}^*$-action on $M_{\rm Dol}(X,r)$ is given by:
\begin{align*}
t: M_{\mathrm{Dol}}(X,r)&\longrightarrow M_{\mathrm{Dol}}(X,r)\\
((E,\bar\partial_E),\theta)&\longmapsto((E,\bar\partial_E),t\theta),
\end{align*}
which plays a crucial role in studying the moduli space. Due to the Simpson--Mochizuki correspondence, we can construct  a new  action on $M_{\mathrm{Dol}}(X,r)$ as follows. Fixing some $(\lambda,t) \in\mathbb{C}\times\mathbb{C}^*$, for any stable  Higgs bundle $((E,\bar{\partial}_E),\theta)\in M_{\mathrm{Dol}}(X,r)$ with a pluri-harmonic metric $h$, we have a stable  $\lambda$-flat bundle $((E,d^{\prime\prime}_E=\bar{\partial}_E+\lambda\theta_h^\dagger),D^\lambda=\lambda\partial_{E,h}+\theta)\in M^\lambda_{\mathrm{Hod}}(X,r)$, and a stable  $\lambda^\prime$-flat bundle $((E,d^{\prime\prime}_E),D^{\lambda^\prime}=t\lambda\partial_{E,h}+t\theta)\in M^{\lambda^\prime}_{\mathrm{Hod}}(X,r)$ for $\lambda^\prime =t\lambda$, the latter one admits a pluri-harmonic metric $h_t$, which gives rise to a stable
 Higgs bundle $((E,\bar\partial_{E,h_t}),\theta_{h_t})\in M_{\mathrm{Dol}}(X,r)$ by the  Simpson--Mochizuki correspondence. We conclude the above process in the following:
 \begin{align*}
 M_{\mathrm{Dol}}(X,r)&\xrightarrow{\hspace*{1.7cm}} M_{\rm Hod}^\lambda(X,r)\xrightarrow{\hspace*{3cm}} M_{\rm Hod}^{\lambda'}(X,r)\xrightarrow{\hspace*{2.1cm}} M_{\rm Dol}(X,r)\\
((E,\bar\partial_E),\theta)&\longmapsto((E,\bar{\partial}_E+\lambda\theta_h^\dagger),\lambda\partial_{E,h}+\theta)\longmapsto((E,\bar{\partial}_E+\lambda\theta_h^\dagger),t\lambda\partial_{E,h}+t\theta))\longmapsto((E,\bar\partial_{E,h_t}),\theta_{h_t}).
 \end{align*}
 As a summary, the Simpson--Mochizuki correspondence provides a dynamical system (i.e. a smooth  self-map) with two-parameters on the Dolbeault moduli space $M_{\mathrm{Dol}}(X,r)$
 \begin{align*}
   \psi_{(\lambda,t)}:M_{\mathrm{Dol}}(X,r)&\longrightarrow M_{\mathrm{Dol}}(X,r)\\
((E,\bar{\partial}_E),\theta)&\longmapsto((E,\bar\partial_{E,h_t}),\theta_{h_t}),
 \end{align*}
 and we also call it  the $(\lambda,t)$-action.  
 
\begin{remark}\
\begin{enumerate}
\item Clearly,  $\psi_{(\lambda,t)}$  can also be defined on  $\mathbb{M}_{\mathrm{Dol}}(X,r)$ as a continuous self-map.
\item The similar construction proceeding from $\mathbb{M}_{\mathrm{dR}}(X,r)$ provides a dynamical system on $\mathbb{M}_{\mathrm{dR}}(X,r)$.
\end{enumerate}
\end{remark} 
 
The following several facts are very obvious.

\begin{proposition}\
 \begin{enumerate}
\item $\psi_{(0,t)}$ is the usual $\mathbb{C}^*$-action by $t$ on $M_{\mathrm{Dol}}(X,r)$, and $\psi_{(\lambda,1)}$ is the identity morphism,
  \item  $\psi_{(\lambda t_1,t_2)}\circ \psi_{(\lambda,t_1)}=\psi_{(\lambda,t_1t_2)}$,
  \item A stable vector bundle (with zero Higgs field) with vanishing Chern classes is a  fixed point of $\psi_{(\lambda,t)}$ for any $\lambda\in\mathbb{C},t\in \mathbb{C}^*$.
\end{enumerate}
 \end{proposition}

\subsection{The First Variation}

Let $u=((E,\bar\partial_E),\theta)\in M_{\mathrm{Dol}}(X,r)$ with the pluri-harmonic metric $h$, the tangent space of $M_{\mathrm{Dol}}(X,r)$ at $u$ is given by the hypercohomology $\mathbb{H}^1(\End(E),\tilde\theta)$ of Higgs complex \cite{CS4}
$$
\End(E)\xrightarrow{\tilde\theta\wedge }\End(E)\otimes_{\mathcal{O}_X}\Omega^1_X\xrightarrow{\tilde \theta\wedge }\cdots.
$$
By K\"{a}hler identities, there is an isomorphism
\begin{align*}
  \mathbb{H}^1(\End(E),\tilde\theta)\simeq \mathcal{H}^1(E,\theta)
  :=&\{(\alpha,\beta)\in\Omega^{0,1}_X(\End(E))\oplus\Omega^{1,0}_X(\End(E))\\&\ \ \ :(\widetilde{\partial_{E,h}}+\widetilde{\theta^\dagger_h})(\alpha+\beta)=(\widetilde{\bar\partial_E}+\tilde\theta)(\alpha+\beta)=0\}.
\end{align*}

\begin{definition}[\cite{HH1}] 
The pair $(\alpha,\beta)\in\mathcal{H}^1(E,\theta)$ is called the \emph{infinitesimal deformation} of the  Higgs bundle $((E,\bar\partial_E),\theta)$, in particular, if $\widetilde{\partial_{E,h}}\alpha=\widetilde{\bar\partial_E}\beta=0$, $(\alpha,\beta)$ is called the \emph{holomorphic infinitesimal deformation}.
\end{definition}

 Now we assume $X$ is a  Riemann surface. Consider the family $u(s):=((E,\bar\partial_{E_s}),\theta_s)$ lying in $M_{\mathrm{Dol}}(X,r)$ with parameter $s$ such that
$u(0)=u$ and $\frac{du(s)}{ds}|_{s=0}=(\alpha,\beta)\in \mathcal{H}^1(E,\theta)$. The  pluri-harmonic metric on the Higgs bundle $((E,\bar\partial_{E_s}),\theta_s)$ is denoted by $h(s)$ with $h(0)=h$, and fixing $\lambda,t$, the pluri-harmonic metric on the $\lambda^\prime$-flat bundle $((E,d^{\prime\prime}_{E_s}=\bar\partial_{E_s}+\lambda(\theta_s)^\dagger_{h(s)}),d^\prime_{E_s}=t(\lambda\partial_{E_s,h(s)}+\theta_s))$ is denoted by $h_t(s)$ with $h_t(0)=h_t$, which yields the operators $\delta^\prime_{E_s}:=\delta^\prime_{E,h_t(s)}$ and $\delta^{\prime\prime}_{E_s}:=\delta^{\prime\prime}_{E,h_t(s)}$. There is an integral curve $\gamma$ in $M_{\mathrm{Dol}}(X,r)$ passing through the point $u$ with tangent vector $(\alpha,\beta)$, the $(\lambda,t)$-action maps this curve to another curve $\gamma^\prime$, we can study its  local property at the point $\psi_{(\lambda,t)}(u)$ by calculating the  variations of  $\psi_{(\lambda,t)}$.

 \begin{proposition}
 Assume the original point $u$ and the parameters $\lambda,t$ are chosen to satisfy $h_t=h$, and assume $\frac{du(s)}{ds}|_{s=0}=(\alpha,\beta)$ is a holomorphic infinitesimal deformation, then
$$
\frac{d\psi_{(\lambda,t)}u(s)}{ds}\bigg|_{s=0}=\bigg(\alpha+\frac{\lambda(1-|t|^2)}{1+|t\lambda|^2}\beta^\dagger_h,\frac{t(1+|\lambda|^2)}{1+|t\lambda|^2}\beta\bigg).
$$
\end{proposition}

\begin{proof} 
We write $h_t(s)=h_tH_t(s)$, and $d^\prime_E=d^\prime_{E_0}, d^{\prime\prime}_E=d^{\prime\prime}_{E_0}, \delta^{\prime}_E=\delta^{\prime}_{E_0},\delta^{\prime\prime}_E=\delta^{\prime\prime}_{E_0}$, then  choosing a local $h_t$-unitary frame $\{e_i\}$ of $E$,  we have
\begin{align*}
  \lambda^\prime\partial h_t(s)(e_i,e_j)=& \lambda^\prime \partial h_t(H_t(s)e_i,e_j)
  = h_t(d^\prime_E(H_t(s)e_i),e_j)+ h_t(H_t(s)e_i,\delta^{\prime\prime}_Ee_j)\\
  =&h_t(H_t(s)d^\prime_{E_s}e_i, e_j)+h_t(H_t(s)e_i,\delta^{\prime\prime}_{E_s}e_j),\\
   \bar\partial h_t(s)(e_i,e_j)=& \bar\partial h_t(H_t(s)e_i,e_j)
  = h_t(d^{\prime\prime}_E(H_t(s)e_i),e_j)+ h_t(H_t(s)e_i,\delta^{\prime}_Ee_j)\\
  =&h_t(H_t(s)d^{\prime\prime}_{E_s}e_i, e_j)+h_t(H_t(s)e_i,\delta^{\prime}_{E_s}e_j).
\end{align*}
 Taking derivative with respect to $s$ and evaluating at $s=0$ give rise to
\begin{align*}
   h_t\bigg(\widetilde{d^\prime_E}\Big(\frac{d H_t(s)}{ds}\bigg|_{s=0}\Big)e_i,e_j\bigg)&=h_t\bigg(\frac{d (d^\prime_{E_s})}{ds}\bigg|_{s=0}e_i,e_j\bigg)+h_t(e_i,\frac{d (\delta^{\prime\prime}_{E_s})}{ds}|_{s=0}e_j),\\
  h_t\bigg(\widetilde{d^{\prime\prime}_E}\Big(\frac{d H_t(s)}{ds}\bigg|_{s=0}\Big)e_i,e_j\bigg)&=h_t\bigg(\frac{d (d^{\prime\prime}_{E_s})}{ds}\bigg|_{s=0}e_i,e_j\bigg)+h_t\bigg(e_i,\frac{d (\delta^{\prime}_{E_s})}{ds}\bigg|_{s=0}e_j\bigg),
\end{align*}
which implies that
\begin{align*}
\frac{d (\delta^{\prime}_{E_s})}{ds}\bigg|_{s=0}&=\widetilde{\delta^{\prime}_E}\bigg(\frac{d H_t(s)}{ds}\bigg|_{s=0}\bigg)-\bigg(\frac{d (d^{\prime\prime}_{E_s})}{ds}\bigg|_{s=0}\bigg)^\dagger_{h_t},\\
\frac{d (\delta^{\prime\prime}_{E_s})}{ds}\bigg|_{s=0}&=\widetilde{\delta^{\prime\prime}_E}\bigg(\frac{d H_t(s)}{ds}\bigg|_{s=0}\bigg)-\bigg(\frac{d (d^{\prime}_{E_s})}{ds}\bigg|_{s=0}\bigg)^\dagger_{h_t}.
\end{align*}
On the other hand, from the pluri-harmonicity of $h_t(s)$, namely
$$
[d^{\prime\prime}_{E_s}+\lambda^\prime\delta^{\prime\prime}_{E_s},d^{\prime}_{E_s}-\lambda^\prime\delta^{\prime}_{E_s}]=-\lambda^\prime [d^{\prime\prime}_{E_s},\delta^{\prime}_{E_s}]+\lambda^\prime[d^{\prime}_{E_s},\delta^{\prime\prime}_{E_s}]=0,
$$
it follows that
\begin{align}\label{a}
  \bigg(\widetilde{\delta^\prime_E}\widetilde{d^{\prime\prime}_E}\frac{d H_t(s)}{ds}\bigg|_{s=0}-\widetilde{\delta^{\prime\prime}_E}\widetilde{d^{\prime}_E}\frac{d H_t(s)}{ds}\bigg|_{s=0}\bigg)
  &+\bigg(\widetilde{d^{\prime\prime}_E}\bigg(\frac{d (d^{\prime\prime}_{E_s})}{ds}\bigg|_{s=0}\bigg)^\dagger_{h_t}-\widetilde{d^{\prime}_E}\bigg(\frac{d (d^{\prime}_{E_s})}{ds}\bigg|_{s=0}\bigg)^\dagger_{h_t}\bigg)\nonumber\\
  &-\bigg(\widetilde{\delta^\prime_E}\frac{d (d^{\prime\prime}_{E_s})}{ds}\bigg|_{s=0}-\widetilde{\delta^{\prime\prime}_E}\frac{d (d^{\prime}_{E_s})}{ds}\bigg|_{s=0}\bigg)=0.
\end{align}
Due to  \cite[Proposition 3.2]{CW}, we have
$$
\frac{d (d^{\prime\prime}_{E_s})}{ds}\bigg|_{s=0}=\alpha+\lambda\beta^\dagger_{h},\quad  \frac{d (d^{\prime}_{E_s})}{ds}\bigg|_{s=0}=-\lambda^\prime\alpha^\dagger_{h}+t\beta.
$$
The condition  $h_t=h$ leads to
$$
\delta^\prime_E=\partial_{E,h}-\bar\lambda\theta,\quad  \delta^{\prime\prime}_E=\bar{\lambda^\prime}\bar\partial-\bar t\theta^\dagger_h,
$$
then by \eqref{a}, since $(\alpha,\beta)$ is an  infinitesimal holomorphic deformation, we arrives at
$$
\widetilde{{\mathbb{D}^{\lambda^\prime}_{h_t}}^\star}\widetilde{\mathbb{D}^{\lambda^\prime}}\frac{d H_t(s)}{ds}\bigg|_{s=0}=0,
$$
for which applying the K\"{a}hler identities in Proposition \ref{ki}
implies
 $\widetilde{\mathbb{D}^{\lambda^\prime}}\frac{d H_t(s)}{ds}|_{s=0}=0$. But $\lambda^\prime$-flat bundle $(E, \mathbb{D}^{\lambda^\prime}=d^{\prime}_E+d^{\prime\prime}_E)$ is simple, $\frac{d H_t(s)}{ds}|_{s=0}$ has to be constant.
Therefore, from the calculation of 
$$
\frac{d}{ds}\bigg|_{s=0}\bigg(\frac{1}{1+|\lambda^\prime|^2}(d^{\prime\prime}_{E_s}+\lambda^\prime\delta^{\prime\prime}_{E_s,h_t(s)}),\frac{1}{1+|\lambda^\prime|^2}(d^{\prime}_{E_s}-\lambda^\prime\delta^{\prime}_{E_s,h_t(s)})\bigg),
$$
the desired result immediately follows.
\end{proof}

\subsection{Fixed Points}

In this subsection, we study the fixed points of the dynamical system $\psi_{(\lambda,t)}$. We first introduce the following definitions.

\begin{definition}\
\begin{itemize}
\item[(1)] For a given Higgs bundle $u\in  \mathbb{M}_{\mathrm{Dol}}(X,r)$, one defines the \emph{set of stable parameters} as 
$$
\mathcal{C}_u=\{(\lambda,t)\in \mathbb{C}\times\mathbb{C}^*: \psi_{(\lambda,t)}(u)=u\}.
$$
\item[(2)] For a given pair $(\lambda,t)\in \mathbb{C}\times \mathbb{C}^*$, one defines the \emph{set of fixed points} as 
$$
\mathfrak{Fix}_{(\lambda,t)}=\{u\in \mathbb{M}_{\mathrm{Dol}}(X,r): \psi_{(\lambda,t)}(u)=u\}.
$$
\end{itemize}
\end{definition}

\begin{definition}[\cite{MSWW,TM3}] 
A Higgs bundle  $((E,\bar\partial_E),\theta)$ over $X$ is called
a \emph{decoupled Higgs bundle} if there is a  Hermitian metric $h$  on $E$ satisfying $R(h)=(\partial_{E,h}+\bar\partial_E)^2=0 $ and $[\theta,\theta^\dagger_h]=0$, and in this case, such metric is called a \emph{decoupling metric}.
\end{definition}

\begin{proposition}
Let $X$ be a Riemann surface of genus $g\geq 2$, and let $M_{\mathrm{de}}(X,r)$ be the subset of $M_{\mathrm{Dol}}(X,r)$ consisting of stable decoupled Higgs bundles. Then $M_{\mathrm{de}}(X,r)$ is a connected real  analytic subvariety of $M_{\mathrm{Dol}}(X,r)$ with dimension $3r^2(g-1)+rg+3$.
\end{proposition}

\begin{proof}
It is known that the Hitchin moduli space $M_{\mathrm{Hit}}(X,r)$  defined as the space of irreducible Hitchin pairs  (solutions to Hitchin's self-duality equations with a given Hermitian metric on the complex vector bundle) modulo unitary gauge transformations is diffeomorphic to $M_{\mathrm{Dol}}(X,r)$.  This means $M_{\mathrm{de}}(X,r)$ can be defined as a subset of $M_{\mathrm{Hit}}(X,r)$ consisting of irreducible decoupled Hitchin pairs. The forgetful map $(E,\theta)\mapsto E$ provides a  fibration $M_{\mathrm{de}}\rightarrow B(X,r)$, where $B(X,r)$ is the moduli space  of rank $r$ stable bundles with vanishing the first Chern class over $X$. One locally writes $\theta=\Theta dz$ for an $r\times r$ complex matrix $\Theta$, then the  condition $[\Theta,\Theta^\dagger]=0$ implies that $\Theta$ is a normal matrix, hence the real dimension of the fibers  is given by $r(r+1)g-(r^2-1)=r^2(g-1)+rg+1$. It follows from   the invariance of $\mathbb{C}^*$-action on  $M_{\mathrm{de}}(X,r)$  and the connectedness of $B(X,r)$ that  $M_{\mathrm{de}}(X,r)$ is connected.
\end{proof}

\begin{theorem}\label{9}
Let $X$ be a  Riemann surface and let $u\in  \mathbb{M}_{\mathrm{Dol}}(X,r)$ represents a decoupled Higgs bundle of rank $r$ with  nontrivial Higgs field, then $\mathbb{C}\times\{\mu_l^m, m=0,\cdots,l-1\}\subseteq \mathcal{C}_u\subseteq (\mathbb{C}\times\{\mu_l^m, m=0,\cdots,l-1\})\bigcup\{(\lambda,t)\in \mathbb{C}^*\times\mathbb{C}^*:|t||\lambda|^2=1,|t|\neq1, t=|t|\mu_{l'}^k, k=1,\cdots,l'-1\}$, where $\mu_l=e^{\frac{2\pi i}{l}}, \mu_{l'}=e^{\frac{2\pi i}{l'}}$ for some fixed positive integers $1\leq l\leq r, 2\leq l'\leq r$.
\end{theorem}

\begin{proof}
  \textbf{Case I}: $\lambda\neq0, |t|= 1$.
  
 Let $(E,\mathbb{D}^\lambda,h)$ be a stable    $\lambda$-flat bundle with the  pluri-harmonic metric $h$.  The operators $\delta^\prime_{h_t},\delta^{\prime\prime}_{h_t}, \partial_{h_t}, \bar\partial_{h_t}, \theta_{h_t}, \theta^\dagger_{h_t}$ can be defined via $(\mathbb{D}^\lambda, h_t)$ and $(\mathbb{D}^{\lambda^\prime}, h_t)$, respectively, in order to distinguish them, we add the  subscripts $\lambda, \lambda^\prime$ for them. Then by definition, we have
\begin{align*}
 \delta^\prime_{h_t,\lambda^\prime}=\delta^\prime_{h_t,\lambda},\quad
  \delta^{\prime\prime}_{h_t,\lambda^\prime}=\bar t\delta^{\prime\prime}_{h_t,\lambda},
\end{align*}
hence
\begin{equation}\label{p0}
\begin{aligned}
\bar\partial_{h_t,\lambda^\prime}&=\frac{1}{1+|t\lambda|^2}\bigg(d^{\prime\prime}_E+|t|^2\lambda\delta^{\prime\prime}_{h_t,\lambda}\bigg),\quad
 \partial_{h_t,\lambda^\prime}=\frac{1}{1+|t\lambda|^2}\bigg(|t|^2\bar\lambda d^\prime_E+ \delta^\prime_{h_t,\lambda}\bigg),\\
 \theta^\dagger_{h_t,\lambda^\prime}&=\frac{\bar t}{1+|t\lambda|^2}\bigg(\bar \lambda d^{\prime\prime}_E-\delta^{\prime\prime}_{h_t,\lambda}\bigg),\qquad\ \
 \theta_{h_t,\lambda^\prime}=\frac{ t}{1+|t\lambda|^2}\bigg(d^{\prime}_E-\lambda\delta^{\prime}_{h_t,\lambda}\bigg).
\end{aligned}
\end{equation}
When $|t|=1$, we arrive at
\begin{align*}
   \bar\partial_{h_t,\lambda^\prime}&= \bar\partial_{h_t,\lambda}, \ \ \ \partial_{h_t,\lambda^\prime}= \bar\partial_{h_t,\lambda},\\
   \theta^\dagger_{h_t,\lambda^\prime}&=\bar t\theta_{h_t,\lambda},\  \ \ \theta_{h_t,\lambda^\prime}= t\theta_{h_t,\lambda}.
\end{align*}
It follows from $\bar\partial_{h_t,\lambda^\prime}^2=\bar\partial_{h_t,\lambda^\prime} \theta_{h_t,\lambda^\prime}= \theta_{h_t,\lambda^\prime}\wedge  \theta_{h_t,\lambda^\prime}=0$ that
$\bar\partial_{h_t,\lambda}^2=\bar\partial_{h_t,\lambda} \theta_{h_t,\lambda}= \theta_{h_t,\lambda}\wedge  \theta_{h_t,\lambda}=0$, namely, $h_t$ is also a pluri-harmonic metric on $(E,\mathbb{D}^\lambda)$. Then by the uniqueness of pluri-harmonic metric, we have $h_t=c\cdot h$ for some constant $c$ when $|t|=1$. 

Consequently, the dynamical system $\psi_{(\lambda,t)}$ sends a polystable Higgs bundle $((E,\bar \partial_E ),\theta)$ to another one $((E,\bar \partial_E ),t\theta)$, namely,  $\psi_{(\lambda,t)}$ is just the
usual $S^1$-action by $t$ on $\mathbb{M}_{\mathrm{Dol}}(X,r)$.

Now let $((E,\bar \partial_E ),\theta)$ be a decoupled  Higgs bundle with the Higgs field $\theta$ nonzero. If it is a fixed point of $\psi_{(\lambda,t)}$ for $|t|=1$, then there is a $C^\infty$-automorphism $\mathfrak{g}\in \Aut(E)$ such that
\begin{align}
\mathfrak{g}\widetilde{\bar\partial_E}\mathfrak{g}^{-1}&=0,\\
 \mathfrak{ g}\theta \mathfrak{g}^{-1}&=t\theta.\label{eq5.4}
\end{align}
 Since $(E,\bar \partial_E)$ is  already a polystable  bundle, thus $(E,\bar \partial_E)=\bigoplus_{i=1}^N(E_i,\bar \partial_{E_i})$ for stable bundles $(E_i,\bar \partial_{E_i})$ with vanishing Chern classes, by  the first equation,   $\mathfrak{g}$ must be of the following form
  $$
  \mathfrak{g}=\begin{pmatrix} \ a_1\mathrm{Id}_{E_1}& & & \\    &  &\ddots &\\ & && a_N\mathrm{Id}_{E_N} \end{pmatrix}
  $$
for nonzero constants $a_1,\cdots,a_N$. If there exists $i$ such that  $pr_{E_i}\circ\theta|_{E_i}:E_i\rightarrow E_i\otimes K_X$ is nonzero, where $pr_{E_i}$ denotes the projection onto $E_i\otimes K_X$ of $E\otimes K_X$, then the second equation admits a solution  for $\mathfrak{g}$ exist if and only if  $t=1$. If each $pr_{E_i}\theta|_{E_i}$ vanishes,
since $\theta$ is nonzero and $[\theta,\theta^\dagger_h]=0$, there exist $ i_1\neq i_2\neq\cdots\neq i_l$ for $1\leq i_1,\cdots,i_l\leq N$ such that $pr_{E_{i_{\mu+1}}}\circ\theta|_{E_{i_\mu}}: E_{i_\mu}\rightarrow E_{i_{\mu+1}}\otimes K_X$ for $1\leq \mu\leq l-1$ and $pr_{E_{i_{1}}}\circ\theta|_{E_{i_l}}: E_{i_l}\rightarrow E_{i_{1}}\otimes K_X$ are all nonzero. Therefore by the equation \eqref{eq5.4} we have
\begin{align}\label{19}
  a_{i_1}=ta_{i_2}, a_{i_2}=ta_{i_3},\cdots a_{i_{l-1}}=ta_{i_l},a_{i_l}=ta_{i_1},
\end{align}
 thus $t$ has to be   $l$-roots of units. Moreover  all components  $a_1,\cdots,a_N$ are solved  by a series of equations as the form of \eqref{19}.

 \textbf{Case II}: $\lambda=0, t\in \mathbb{C^*}$.
 
  In this case, $(\lambda,t)$-action is just the scalar multiplication on Higgs field by $t$. The same conclusions as above follows.

\textbf{Case III}:  $\lambda\neq0, |t|\neq 1$.

Let      $((E,\bar\partial_E),\theta)$ be a decouped Higgs  bundle with the Higgs field $\theta$ nonzero  and the decoupling metric $h$. 
One writes
\begin{align*}
  td^\prime_E&=t\lambda(\partial_{E,h}+\frac{t-a}{t\lambda}\theta)+a\theta,\\
  d^{\prime\prime}_E&=\bar\partial_E+\lambda(1-t\bar a)\theta^\dagger_h+t\lambda \bar a\theta^\dagger_h,
\end{align*}
for some $a\in \mathbb{C}$, then $((E,\bar\partial_E+\lambda(1-t\bar a)\theta^\dagger_h),a\theta)$ is a Higgs bundle. Note that $(\partial_{E,h}-\bar\lambda(1-\bar t a)\theta)+(\bar\partial_E+\lambda(1-t\bar a)\theta^\dagger_h)$ is a unitary connection with respect to $h$. If one takes $$a=t\frac{1+|\lambda|^2}{1+|\lambda^\prime|^2},$$
      we find that  $h $  is the pluri-harmonic metric  both for the $\lambda^\prime$-flat bundle $(E,\mathbb{D}^{\lambda^\prime}=td^\prime_E+d^{\prime\prime}_E)$ and the Higgs bundle $((E,\bar\partial_E+\lambda(1-t\bar a)\theta^\dagger_h),a\theta)$. Therefore,  by the uniqueness of pluri-harmonic metric we get 
      $$
      \psi_{(\lambda,t)}((E,\bar\partial_E),\theta)=((E,\bar\partial_E+\lambda(1-t\bar a)\theta^\dagger_h),a\theta).
      $$

Assume  there is a $C^\infty$-automorphism $\mathfrak{g}\in \Aut(E)$ such that
 \begin{align}
    \mathfrak{g}\widetilde{\bar\partial_E}\mathfrak{g}^{-1}-\lambda(1-t\bar a)\theta^\dagger_h&=0,\label{j}\\
    \mathfrak{g}\theta \mathfrak{g}^{-1}-a\theta&=0\label{k}.                                                                                                                                \end{align}
Since $[\theta,\theta_h^\dagger]=0$, over a small neighborhood of some point $x\in X$ with $\theta|_x\neq0$, there is an orthogonal decomposition of $(E,h)$ into Hermitian line bundles as $(E,h)=\bigoplus_{i=1}^r(L_i,h_i)$ such that the Higgs field $\theta$ has the decomposition $\theta=\bigoplus_{i=1}^r\varphi_i\cdot\mathrm{Id}_{L_i}$ with one-forms $\varphi_i$ \cite{TM3}.  Then from the equation \eqref{k} it follows that $|a|=1$, namely $|t||\lambda|^2=1\ (|\lambda|^2\neq1)$, hence $a=\frac{t}{|t|}$.
We consider $n$-times iteration of  the $(\lambda,t)$-action on $((E,\bar\partial_E),\theta)$. The direct calculation shows
$$
\psi^n_{(\lambda,t)}((E,\bar\partial_E),\theta)=\left\{
                                                 \begin{array}{ll}
                                                   ((E,\bar\partial_E+n(\lambda-\frac{1}{\bar\lambda})\theta^\dagger_h),\theta), & \hbox{$t>0,t\neq 1$;} \\
                                                  \bigg((E,\bar\partial_E+(\lambda-\frac{1}{\bar\lambda})\frac{1-(\frac{\bar t}{|t|})^n}{1-\frac{\bar t}{|t|}}\theta^\dagger_h),(\frac{t}{|t|})^n\theta\bigg) , & \hbox{\textrm{other cases.}}
                                                 \end{array}
                                               \right.
$$
By assumption, the limit $\lim\limits_{n\rightarrow \infty}\psi^n_{(\lambda,t)}((E,\bar\partial_E),\theta)$ lies in the isomorphism class of $((E,\bar\partial_E),\theta)$, hence $t$ cannot be a positive real number. For the other cases, writing $t=|t|e^{i\alpha}$, we must have $e^{i\alpha nl'}=1$ for any positive integer $n$, where $2\leq l'\leq r$ is a fixed positive integer, therefore, $\alpha=\frac{2k\pi}{l'}$ for some $k\in\{1,\cdots,l'-1\}$.

Combining the three cases together, we complete the proof the theorem.
\end{proof}

\begin{corollary}\label{coro4}\
\begin{enumerate}
   \item If $\Tr(\theta)$ is nonzero at some point $x\in X$, then $\mathcal{C}_u=\mathbb{C}\times \{1\}$.
   \item  Fixing $(\lambda,t)\in\mathbb{C}\times\mathbb{C}^*$ with $t|\lambda|^2\neq1$ and $t\neq 1$,  let $\mathfrak{Fix}^{\mathrm{de}}_{(\lambda,t)}=\mathfrak{Fix}_{(\lambda,t)}\bigcap M_{\mathrm{de}}(X,r)$, then   $B(X,r)$ as a subvariety of $M_{\mathrm{Dol}}(X,r)$ is a connected component of $\mathfrak{Fix}^{\mathrm{de}}_{(\lambda,t)}$.
\end{enumerate}
\end{corollary}

\begin{proof} (1) is obvious. To show  (2), we consider a sequence $\{(E_n,\theta_n)\}$ of stable decoupled Higgs bundles lying in $\mathfrak{Fix}^{\mathrm{de}}_{(\lambda,t)}\backslash B(X,r)$ parameterized by positive integers $n\in [N,\infty)$ for a large   $N$ such that $\lim\limits_{n\rightarrow\infty}(E_n,\theta_n)=(E_\infty,0)\in B(X,r)$. For each $(E_n,\theta_n)$ , there is a $C^\infty$-automorphism $\mathfrak{g}_n\in \Aut(E)$ satisfying the equations \eqref{j} and \eqref{k}. Since  $[\theta_n,(\theta_n)^\dag_{h_n}]=0$ for the decoupling metric $h_n$, by equation \eqref{k} ($a\neq 1$), for any $n\in [N,\infty)$  the automorphism $\mathfrak{g}_n$ locally has a matrix form as
$\left(
   \begin{array}{cc}
     A_n & 0 \\
     0& B_n \\
   \end{array}
 \right)
$, where all diagonal elements of the nonzero matrix $B_n$ are zero. On the other hand, from equation \eqref{j} it follows that $\mathfrak{g}_\infty=\lim\limits_{n\rightarrow\infty}\mathfrak{g}_n$ is exactly $c\cdot \mathrm{Id}_{E_\infty}$ for some constant $c\in \mathbb{C}^*$, which is a contradiction. Therefore, the desired sequence does not exist, the conclusion follows.
\end{proof}

\begin{remark} 
Studying  $M_{\mathrm{de}}(X,r)$ and $\mathfrak{Fix}^{\mathrm{de}}_{(\lambda,t)}$ is an interesting and hard problem. For example, what are the smooth (or singular) points of $M_{\mathrm{de}}(X,r)$, and dose there exist any other connected components of $\mathfrak{Fix}^{\mathrm{de}}_{(\lambda,t)}$ except $B(X,r)$?
\end{remark}

\begin{theorem}\label{cv}
Let $\mathfrak{Fix}=\bigcap\limits_{(\lambda,t)\in \mathbb{C}^*\times \mathbb{C}^*}\mathfrak{Fix}_{(\lambda,t)}$. Then $\mathfrak{Fix}$ consists of the set of complex variations of Hodge structure.
\end{theorem}

\begin{proof}
Let $\overline{\mathfrak{Fix}}=\bigcap_{(\lambda,t)\in \mathbb{C}\times \mathbb{C}^*}\mathfrak{Fix}_{(\lambda,t)}$. We first show  that $\overline{\mathfrak{Fix}}$ consists of the complex variations of Hodge structure. Consider a complex variation of Hodge structure $u=((E,\bar\partial_E),\theta)$  as 
\begin{equation*} 
\bar\partial_E=\begin{pmatrix} \bar\partial_{E_1}& & \\ &\ddots  &   \\   & &  \bar\partial_{E_k} 
\end{pmatrix}, \quad 
\theta= \begin{pmatrix} 0 & & & \\ \theta_1& 0 & &   \\   & \ddots &\ddots &\\ & & \theta_{k-1}& 0 \end{pmatrix},
\end{equation*}
we only need to show $\mathcal{C}_u=\mathbb{C}\times C^*$.
By  the  pluri-harmonic metric $h$ on $((E,\bar\partial_E),\theta)$ which makes the splitting $E=\bigoplus_{i=1}^kE_i$ being orthogonal, assuming $\lambda\neq 0$, it produces two flat bundles $((E,\bar\partial_E^\prime),\nabla^\prime)$ and $((E,\bar\partial_E^{\prime\prime}),\nabla^{\prime\prime})$ given by
\begin{align*}
\nabla^\prime&= \begin{pmatrix} \ \partial_{E_1,h}& & & \\ \lambda^{-1}\theta_1& \ \partial_{E_2,h} & &   \\   & \ddots &\ddots &\\ & & \lambda^{-1}\theta_{k-1}& \partial_{E_k,h} \end{pmatrix},\qquad\ \ \
\bar \partial_E^\prime=\begin{pmatrix} \bar\partial_{E_1}&\lambda(\theta_1)^\dagger_{h} & & \\   & \ddots &\ddots &\\ & & \bar\partial_{E_{k-1}}& \lambda(\theta_{k-1})^\dagger_h\\ & & & \bar\partial_{E_{k}} \end{pmatrix},\\
\nabla^{\prime\prime}&= \begin{pmatrix} \ \partial_{E_1,h}& & & \\ (t\lambda)^{-1}\theta_1& \ \partial_{E_2,h} & &   \\   & \ddots &\ddots &\\ & & (t\lambda^{-1})\theta_{k-1}& \partial_{E_k,h} \end{pmatrix},\quad
\bar \partial_E^{\prime\prime}=\begin{pmatrix} \bar\partial_{E_1}&t\lambda(\theta_1)^\dagger_{h} & & \\   & \ddots &\ddots &\\ & & \bar\partial_{E_{k-1}}& t\lambda(\theta_{k-1})^\dagger_h\\ & & & \bar\partial_{E_{k}} \end{pmatrix}.
\end{align*}
If these two flat bundles are equivalent to each other, then there is  a $C^\infty$-automorphism $\mathfrak{g}\in \Aut(E)$ such that
\begin{align*}
 \mathfrak{g} \begin{pmatrix}  \widetilde{ \partial_{E_1,h}}& & & \\ & \ \widetilde{\partial_{E_2,h}} & &   \\   &  &\ddots &\\ & && \widetilde{\partial_{E_k,h}} \end{pmatrix} \mathfrak{g}^{-1}+\lambda^{-1}\mathfrak{g}\begin{pmatrix} \ 0& & & \\ \theta_1& \ 0 & &   \\   & \ddots &\ddots &\\ & & \theta_{k-1}&0 \end{pmatrix}\mathfrak{g}^{-1}&=(t\lambda)^{-1}\begin{pmatrix} \ 0& & & \\ \theta_1& \ 0 & &   \\   & \ddots &\ddots &\\ & & \theta_{k-1}&0\\ \end{pmatrix},\\
 \mathfrak{g} \begin{pmatrix} \ \widetilde{ \bar\partial_{E_1}}& & & \\ & \ \widetilde{\bar\partial_{E_2}} & &   \\   &  &\ddots &\\ & && \widetilde{\bar\partial_{E_k}}\\ \end{pmatrix} \mathfrak{g}^{-1}+\lambda \mathfrak{g}\begin{pmatrix} 0&(\theta_1)^\dagger_{h} & & \\   & \ddots &\ddots &\\ & & 0& (\theta_{k-1})^\dagger_h\\ & & & 0\\ \end{pmatrix}\mathfrak{g}^{-1}
 &=t\lambda\begin{pmatrix} 0&(\theta_1)^\dagger_{h} & & \\   & \ddots &\ddots &\\ & & 0& (\theta_{k-1})^\dagger_h\\ & & & 0\\ \end{pmatrix}.
\end{align*}
Obviously, the above equations have a solution
$$
\mathfrak{g}=\begin{pmatrix} \ \mathrm{Id}_{E_1}& & & \\ & \ t^{-1}\mathrm{Id}_{E_2} & &   \\   &  &\ddots &\\ & && t^{-k+1}\mathrm{Id}_{E_k}\\ \end{pmatrix}.
$$
It immediately follows that $\psi_{(\lambda,t)}((E,\bar\partial_E),\theta)=((E,\bar\partial_E),\theta)$ for any $(\lambda,t)\in \mathbb{C}\times\mathbb{C}^*$.

Next we show that $\overline{\mathfrak{Fix}}={\mathfrak{Fix}}$. Assume $((E,\bar\partial_E),\theta)$  lies in $\mathfrak{Fix}$, then we have
\begin{align}\label{5}
  \lim\limits_{t\rightarrow0}\lim\limits_{\lambda\rightarrow0}\psi_{(\lambda,t)}  ((E,\bar\partial_E),\theta)=\lim\limits_{ t\rightarrow0}   ((E,\bar\partial_E),\theta)= ((E,\bar\partial_E),\theta).
\end{align}
On the other hand, let $h$ be a pluri-harmonic metric on  $((E,\bar\partial_E),\theta)$ and $h_t$ be the  pluri-harmonic metric on $\psi_{\lambda,t}((E,\bar\partial_E),\theta)$.  Writing $h_t=h\cdot s$ with  $s=e^\chi$ for $\chi\in \End(E)$,  the direct calculation shows that  the image of $(\lambda,t)$-action on $((E,\bar\partial_E),\theta)$ is given by $ \psi_{(\lambda,t)}((E,\bar\partial_E),\theta)=((E,\bar\partial^\prime_E),\theta^\prime)$, where
\begin{align*}
\bar\partial^\prime_E(\lambda,t)=&\bar\partial_E+\frac{\lambda(1-|t|^2)}{1+|t\lambda|^2}\theta^\dagger_h+\frac{\lambda|t|^2}{1+|t\lambda|^2}s^{-1}(\bar
  \lambda\widetilde{\bar\partial_E}-\widetilde{\theta^\dagger_h})s,\\
 \theta^\prime(\lambda,t)= &\frac{t(1+|\lambda|^2)}{1+|t\lambda|^2}\theta-\frac{\lambda t}{1+|t\lambda|^2}s^{-1}(\widetilde{\partial_{E,h}}-\bar\lambda\widetilde{\theta})s.
\end{align*}
The condition $\bar\partial^\prime_E\theta^\prime=0$ gives rise to a equation satisfied by $s$.
Then  we immediately find that
 \begin{align}\label{6}
  \lim\limits_{t\rightarrow0}\lim\limits_{\lambda\rightarrow0}\psi_{(\lambda,t)}  ((E,\bar\partial_E),\theta)= \lim\limits_{ t\rightarrow0}\psi_{(0,t)} ((E,\bar\partial_E),\theta)= \lim\limits_{ t\rightarrow0} ((E,\bar\partial_E),t\theta).
 \end{align}
Comparing \eqref{5} with \eqref{6} implies $((E,\bar\partial_E),\theta)$ is a    complex variation of Hodge structure.
\end{proof}

\subsection{Asymptotic Behaviour}

In this subsection, $X$ is always assumed to be a Riemann surface.  We first recall Simpson's beautiful work on the limits of $\mathbb{C}^*$-action on the Hodge moduli space  $\mathbb{M}_{\mathrm{Hod}}(X,r)$ (for more details, see \cite{CS9,Hua,HH3}).

\begin{definition}[\cite{CS9}]
 Let $E$ be a holomorphic vector bundle over a  Riemann surface $X$ with a holomorphic flat connection $\nabla:E\rightarrow E\otimes_{\mathcal{O}_X}K_X$, where $K_X$ denotes the canonical line bundle over $X$. A decreasing filtration $\{F^\bullet\}$ of $E$ by strict subbundles
$$
E=F^0\supset F^1\supset\cdots\supset F^k=0
$$
is called a \emph{Simpson filtration} if it satisfies the following two conditions:
\begin{itemize}
  \item Griffiths transversality: $\nabla: F^p\rightarrow F^{p-1}\otimes_{\mathcal{O}_X}\Omega^1_X$,
  \item graded-semistability: the associated graded Higgs bundle $(\mathrm{Gr}_F(E),\mathrm{Gr}_F(\nabla))$, where $\mathrm{Gr}_F(E)=\bigoplus_pE^p$ with $E^p=F^p/F^{p-1}$ and $\mathrm{Gr}_F(\nabla)=\bigoplus_p\theta^p$ with $\theta^p: E^p\rightarrow E^{p-1}\otimes_{\mathcal{O}_X}K_X$ induced from $\nabla$, is a semistable Higgs bundle.
\end{itemize}
\end{definition}

\begin{theorem}[{\cite[\textrm{Theorem 2.5, Lemma 4.1, Corallary 4.2, Proposition 4.3}]{CS9}}]\label{s}
Let $(E,\nabla)$ be a flat bundle over a  Riemann surface  $X$.
\begin{enumerate}
   \item There exist  Simpson filtrations $\{F^\bullet\}$ on $(E,\nabla)$.
  
   \item Let $\{F_1^\bullet\}$, $\{F_2^\bullet\}$ be two Simpson filtrations on $(E,\nabla)$, then the associated graded Higgs bundles $(\mathrm{Gr}_{F_1}(E),\mathrm{Gr}_{F_1}(\nabla))$ and $(\mathrm{Gr}_{F_2}(E),\mathrm{Gr}_{F_2}(\nabla))$ are $S$-equivalent.
  
   \item $(\mathrm{Gr}_F(E),\mathrm{Gr}_F(\nabla))$ is a stable Higgs bundle iff the Simpson filtration is unique.
  
   \item $\lim\limits_{t\rightarrow 0}(E,t\cdot\nabla)=(\mathrm{Gr}_F(E),\mathrm{Gr}_F(\nabla))$.
\end{enumerate}
\end{theorem}

Now we apply Simpson's theorem to study the asymptotic behaviour of the dynamical system $\psi_{(\lambda,t)}$. We first introduce the following notations.

\begin{definition}\label{ds}
Given a Higgs bundle $((E,\bar\partial_E),\theta)\in \mathbb{M}_{\mathrm{Dol}}(X,r)$,  we define the following five  limits:
\begin{enumerate}
\item[(1)] $\psi_{\underline{(0,0)}}((E,\bar\partial_E),\theta)
   :=\lim\limits_{t\rightarrow 0}\psi_{(0,t)}((E,\bar\partial_E),\theta),$
   \smallskip
   \item[(2)] $\psi^{\underline{(0,0)}}((E,\bar\partial_E),\theta)
   :=\lim\limits_{t\rightarrow 0}\lim\limits_{\lambda\rightarrow 0}\psi_{(\lambda,t)}((E,\bar\partial_E),\theta),$
   \smallskip
   \item[(3)] $\psi_{\overline{(0,0)}}((E,\bar\partial_E),\theta)
   :=\lim\limits_{\lambda\rightarrow 0}\psi_{(\lambda,0)}((E,\bar\partial_E),\theta)$,
   \smallskip
   \item[(4)] $\psi^{\overline{(0,0)}}((E,\bar\partial_E),\theta)
   :=\lim\limits_{\lambda\rightarrow 0}\lim\limits_{t\rightarrow 0}\psi_{(\lambda,t)}((E,\bar\partial_E),\theta),$
   \smallskip
   \item[(5)] $\psi_{(0,0)}((E,\bar\partial_E),\theta)
   :=\lim\limits_{(\lambda,t)\rightarrow (0,0)}\psi_{(\lambda,t)}((E,\bar\partial_E),\theta),$
   \end{enumerate}
where $\psi_{(\lambda,0)}$ is defined by By Simpson's theorem, namely
\begin{align*}
  \psi_{(\lambda,0)}((E,\bar\partial_E),\theta)
   =&\lim\limits_{ t\rightarrow 0}((E,\bar\partial_E+\lambda\theta^\dagger_h),t\partial_{E,h}+t\lambda^{-1}\theta)\\
   =&(\mathrm{Gr}_{F_\lambda}(E_\lambda),\mathrm{Gr}_{F_\lambda}(\nabla_\lambda),
 \end{align*}
 with $h$ being a pluri-harmonic metric on $((E,\bar\partial_E),\theta)$, $(E_\lambda,\nabla_\lambda)=((E,\bar\partial_E+\lambda\theta^\dagger_h),\partial_{E,h}+\lambda^{-1}\theta)$, and $\{F^\bullet_\lambda\}$ standing for  a Simpson filtration on $(E_\lambda,\nabla_\lambda)$. \end{definition}

\begin{remark}
The first two limits have been used in the proof of Theorem \ref{cv}, and we have showed  that 
$$
\psi_{\underline{(0,0)}}((E,\bar\partial_E),\theta)=\psi^{\underline{(0,0)}}((E,\bar\partial_E),\theta)= \lim\limits_{t\rightarrow 0}((E,\bar\partial_E),t\theta).
$$  
In general, it is not clear whether the last three limits   exist,  secondly, we also do not know whether these  limits  coincide if they all exist.
\end{remark}

\begin{proposition}
 If for a given Higgs bundle $((E,\bar\partial_E),\theta)\in \mathbb{M}_{\mathrm{Dol}}(X,r)$, the limit $ \psi^{\overline{(0,0)}}((E,\bar\partial_E),\theta)$ (or $ \psi_{(0,0)}((E,\bar\partial_E),\theta)$) exists, then it must be a complex variation of Hodge structure.
\end{proposition}

\begin{proof}
Let $ \psi^{\overline{(0,0)}}((E,\bar\partial_E),\theta)=((E,\bar\partial^\prime_E),\theta^\prime)$, then we calculate
\begin{align*}
 \lim\limits_{\tilde t\rightarrow0}\lim\limits_{\lambda\rightarrow0}\lim\limits_{ t\rightarrow0}\psi_{(\lambda t,\tilde t)}((E,\bar\partial^\prime_E),\theta^\prime)
 = & \lim\limits_{\tilde t\rightarrow0}\lim\limits_{\lambda\rightarrow0}\lim\limits_{ t\rightarrow0}\psi_{(\lambda t,\tilde t)}\circ \psi_{(\lambda,t)}((E,\bar\partial_E),\theta)\\
  =& \lim\limits_{\tilde t\rightarrow0}\lim\limits_{\lambda\rightarrow0}\lim\limits_{ t\rightarrow0}\psi_{(\lambda,\tilde t t)}((E,\bar\partial_E),\theta)\\
  =& \lim\limits_{\tilde t\rightarrow0}((E,\bar\partial^\prime_E),\theta^\prime)=((E,\bar\partial^\prime_E),\theta^\prime),
\end{align*}
on the other hand, we have
\begin{align*}
 \lim\limits_{\tilde t\rightarrow0}\lim\limits_{\lambda\rightarrow0}\lim\limits_{ t\rightarrow0}\psi_{(\lambda t,\tilde t)}((E,\bar\partial^\prime_E),\theta^\prime)
 &= \lim\limits_{\tilde t\rightarrow0}\psi_{(0,\tilde t)}((E,\bar\partial^\prime_E),\theta^\prime)\\
& =\lim\limits_{\tilde t\rightarrow0}((E,\bar\partial^\prime_E),\tilde t\theta^\prime).
\end{align*}
Comparing these two results, we find that $((E,\bar\partial'_E),\theta')$ has to be a complex variation of Hodge structure.
\end{proof}

\begin{theorem}\label{nb}
Let $X$ be  a  Riemann surface.
\begin{enumerate}
\item If $((E,\bar\partial_E),\theta)\in \mathbb{M}_{\mathrm{Dol}}(X,r)$ is a complex variation of Hodge structure or a decoupled Higgs bundle, then the above  limits exist and coincide.
\item Let $((E,\bar\partial_E),\theta)\in \mathbb{M}_{\mathrm{Dol}}(X,2)$ and  assume the maximal destabilizing subbundle of $(E,\bar\partial_E)$ is preserved by $\theta^\dagger_h$ for the pluri-harmonic metric $h$ on $((E,\bar\partial_E),\theta)$, then  the limit $ \psi_{\overline{(0,0)}}((E,\bar\partial_E),\theta)$ exists,  and it  coincides with the limit $ \psi_{\underline{(0,0)}}((E,\bar\partial_E),\theta)$.
\item Let  $((E,\bar\partial_E),\theta)\in M_{\mathrm{Dol}}(X,r)$, then the limit  $\lim\limits_{\lambda\rightarrow 0}\psi_{(\lambda,0)}((E,\bar\partial_E),\lambda\theta)$ exists, and it  coincides with the limit $ \psi_{\underline{(0,0)}}((E,\bar\partial_E),\theta)$.
\end{enumerate}

\end{theorem}

\begin{proof}
(1) i) Let $((E,\bar\partial_E),\theta)\in \mathbb{M}_{\mathrm{Dol}}(X,r)$  be a complex variation of Hodge structure.
Since it is a fixed point of $(\lambda,t)$-action for any $(\lambda,t)\in \mathbb{C}\times \mathbb{C}^*$ by Theorem \ref{cv}, we have
$$ 
\psi_{\underline{(0,0)}}((E,\bar\partial_E),\theta)= \psi^{\underline{(0,0)}}((E,\bar\partial_E),\theta)=\psi^{\overline{(0,0)}}((E,\bar\partial_E),\theta)=\psi_{(0,0)}((E,\bar\partial_E),\theta)=((E,\bar\partial_E),\theta).
$$
Hence we only need to show $\psi_{\overline{(0,0)}}((E,\bar\partial_E),\theta)=((E,\bar\partial_E),\theta)$.  For $\lambda\neq 0$, we write  $(E,\bar\partial_E)=\bigoplus_{i=1}^k(E_i, \bar\partial_{E_i}),\theta=\bigoplus_{i=1}^{k-1}\theta_i$ for  $\theta_i:E_i\rightarrow E_{i+1}\otimes K_X$,
then by  virtue of the pluri-harmonic metric $h$ on  $((E,\bar\partial_E),\theta))$, we have  a holomorphic flat connection
$$
\nabla= \begin{pmatrix} \partial_{E_1,h}& & & \\ \lambda^{-1}\theta_1& \partial_{E_2,h} & &   \\   & \ddots &\ddots &\\ & & \lambda^{-1}\theta_{k-1}& \partial_{E_k,h}\\ \end{pmatrix}
$$
 with respect to the holomorphic structure
 $$\bar \partial_E^\prime=\begin{pmatrix} \bar\partial_{E_1}&\lambda(\theta_1)^\dagger_{h} & & \\   & \ddots &\ddots &\\ & & \bar\partial_{E_{k-1}}& \lambda(\theta_{k-1})^\dagger_h\\ & & & \bar\partial_{E_{k}}\\ \end{pmatrix}.$$
There is  a Simpson filtration $\{F^\bullet\}$ on $((E,\bar \partial_E^\prime),\nabla)$ given by
 $\{F^p=\bigoplus_{i=1}^{k-p} E_i\}_{0\leq p\leq k-1}$ since one easily checks that
 $$\nabla F^p\subset F^{p-1}\otimes K_X, \ \ \bar \partial_E^\prime F^p=0.$$
 It follows that
 $\psi_{(\lambda,0)}((E,\bar\partial_E),\theta)=((E,\bar\partial_E),\lambda^{-1}\theta))$
from Simpson's theorem.
Therefore, $\psi_{\overline{(0,0)}}((E,\bar\partial_E),\theta)=\lim\limits_{\lambda\rightarrow 0}((E,\bar\partial_E),\lambda^{-1}\theta)=((E,\bar\partial_E),\theta)$.

ii) Let $((E,\bar\partial_E),\theta))\in \mathbb{M}_{\mathrm{Dol}}(X,r)$  be a decoupled Higgs bundle with  decoupling metric $h$. We can assume $\theta$ is nonzero. We have seen that $ \psi_{\underline{(0,0)}}((E,\bar\partial_E),\theta)= (E,\bar\partial_E)$,
meanwhile we can  also calculate the limits
    $$ \psi_{\overline{(0,0)}}((E,\bar\partial_E),\theta)=\lim\limits_{\lambda\rightarrow 0}(E,\bar\partial_E+\lambda\theta^\dagger_h)=(E,\bar\partial_E),$$
 and
\begin{align*}
  &\lim\limits_{(\lambda,t)\rightarrow (0,0)}\bigg((E,\bar\partial_E+\frac{\lambda(1-|t|^2)}{1+|t\lambda|^2}\theta^\dagger_h),\frac{t(1+|\lambda|^2)}{1+|t\lambda|^2}\theta\bigg)\\
 =& \lim\limits_{\lambda\rightarrow 0}\lim\limits_{t\rightarrow 0}\bigg((E,\bar\partial_E+\frac{\lambda(1-|t|^2)}{1+|t\lambda|^2}\theta^\dagger_h),\frac{t(1+|\lambda|^2)}{1+|t\lambda|^2}\theta\bigg)=\lim\limits_{\lambda\rightarrow 0}(E,\bar\partial_E+\lambda\theta^\dagger_h)\\
 =&\lim\limits_{t\rightarrow 0}\lim\limits_{\lambda\rightarrow 0}\bigg((E,\bar\partial_E+\frac{\lambda(1-|t|^2)}{1+|t\lambda|^2}\theta^\dagger_h),\frac{t(1+|\lambda|^2)}{1+|t\lambda|^2}\theta\bigg)=\lim\limits_{t\rightarrow 0}((E,\bar\partial_E),t\theta)\\
 =&(E,\bar\partial_E).
\end{align*}
Consequently,  $\psi^{\underline{(0,0)}}((E,\bar\partial_E),\theta)=\psi^{\overline{(0,0)}}((E,\bar\partial_E),\theta)=\psi_{{(0,0)}}((E,\bar\partial_E),\theta)=(E,\bar\partial_E)$.

(2) Consider a family of flat bundles $((E,\bar\partial_E+\lambda\theta^\dagger_h),\partial_{E,h}+\lambda^{-1}\theta)$. It is divided into two cases.

 i)
 Assume $(E,\bar\partial_E+\lambda\theta^\dagger_h)$  are  non-semistable over some small deleted neighborhood $U$ of $\lambda=0$. Let $L$ be the maximal destabilizing subbundle of $(E,\bar\partial_E)$, and  $L^\bot$ be the orthogonal complement of $L$ in $E$ with respect to the pluri-harmonic metric $h$, namely there are $C^\infty$-decompositions $E\simeq L\oplus L^\bot\simeq L\oplus E/L$.  With respect to the above decomposition, we write
$$
\bar\partial_E=\left(
                   \begin{array}{cc}
                     \bar\partial_1 & \alpha \\
                     0 &  \bar\partial_2 \\
                   \end{array}
                 \right), \ \ \theta=\left(
                   \begin{array}{cc}
                     \theta_1 & 0 \\
                     \beta &  \theta_2 \\
                   \end{array}
                 \right),
$$ 
where $\beta$ must be non-zero and satisfies $\widetilde{\bar\partial} _2\beta=0$. By assumption $L$ is preserved by $\theta^\dagger_h$.
Since the Simpson filtration exactly coincides with the Harder--Narasimhan filtration for the case of rank  $r=2$, we get
$$
\psi_{\overline{(0,0)}}((E,\bar\partial_E),\theta)=\lim\limits_{\lambda\rightarrow0}\bigg((E,\left(
                   \begin{array}{cc}
                     \bar\partial_1+\lambda\bar\theta_1 & 0 \\
                     0 &  \bar\partial_2+\lambda\bar\theta_2 \\
                   \end{array}
                 \right)), \left(
                   \begin{array}{cc}
                     0 & 0 \\
                     \bar\alpha+\lambda^{-1}\beta &  0 \\
                   \end{array}
                 \right)\bigg).$$
                 Choosing a $C^\infty$-automorphism $\mathfrak{g}=\left(
                                                                    \begin{array}{cc}
                                                                      1 & 0 \\
                                                                      0 & \lambda \\
                                                                    \end{array}
                                                                  \right)\in\Aut(E)
                 $, from the  identities \begin{align*}
                                   \mathfrak{g}\circ\left(
                   \begin{array}{cc}
                     \bar\partial_1+\lambda\bar\theta_1 & 0 \\
                     0 &  \bar\partial_2+\lambda\bar\theta_2 \\
                   \end{array}
                 \right)\circ\mathfrak{g}^{-1}&=\left(
                   \begin{array}{cc}
                     \bar\partial_1+\lambda\bar\theta_1 & 0 \\
                    0  &  \bar\partial_2+\lambda\bar\theta_2 \\
                   \end{array}
                 \right),\\
                  \mathfrak{g}\circ\left(
                   \begin{array}{cc}
                     0 & 0 \\
                     \bar\alpha+\lambda^{-1}\beta &  0 \\
                   \end{array}
                 \right)\circ\mathfrak{g}^{-1}&=\left(
                   \begin{array}{cc}
                     0 & 0 \\
                     \lambda\bar\alpha+\beta &  0 \\
                   \end{array}
                 \right)
                 \end{align*}
  it follows that    \begin{align*}
                     \psi_{\overline{(0,0)}}((E,\bar\partial_E),\theta)  &=((E,\left(
                   \begin{array}{cc}
                     \bar\partial_1 & 0 \\
                    0 &  \bar\partial_2 \\
                   \end{array}
                 \right), \left(
                   \begin{array}{cc}
                     0 & 0 \\
                     \beta &  0 \\
                   \end{array}
                 \right))\\
              &=   \lim\limits_{t\rightarrow0}((E,\bar \partial_E),t\theta),
                     \end{align*}
        thus $\psi_{\underline{(0,0)}}((E,\bar\partial_E),\theta)=\psi_{\overline{(0,0)}}((E,\bar\partial_E),\theta)$.

        ii) Assume  $(E,\bar\partial_E+\lambda\theta^\dagger_h)$  are  semistable over some small deleted neighborhood $U$ of $\lambda=0$. Then $(E,\bar\partial_E)$ is also a semistable bundle. Otherwise, by our assumption,   the maximal destabilizing subbundle $L$ of $(E,\bar\partial_E)$ is also that of $(E,\bar\partial_E+\lambda\theta^\dagger_h)$, which contradicts the semistability of $(E,\bar\partial_E+\lambda\theta^\dagger_h)$.  Therefore, we have
        $$ \psi_{\overline{(0,0)}}((E,\bar\partial_E),\theta)=\lim\limits_{\lambda\rightarrow0}(E,\bar\partial_E+\lambda\theta^\dagger_h) =(E,\bar\partial_E)=\psi_{\underline{(0,0)}}((E,\bar\partial_E),\theta).$$

        (3) It follows from the calculation of so-called conformal limit in \cite{Gai,DFKMMN,CW}. Indeed, the limit $\lim\limits_{c\rightarrow0}((E,\bar\partial_{E}+|c|^2\theta^\dagger_{h_c}),\partial_{h_c}+\theta)=((E,\bar\partial_E^\prime),\nabla^\prime)$ exists as a flat bundle, where $h_c$ is a pluriharmobic metric on the Higgs bundle $((E,\bar\partial_E),c\theta)$, and it  satisfies $\lim\limits_{t\rightarrow0}((E,\bar\partial_E^\prime),t\nabla^\prime)=\lim\limits_{c\rightarrow0}((E,\bar\partial_E),c\theta)$.
\end{proof}


\begin{thebibliography}{9}

\bibitem{Ari} 
D. Arinkin, On $\lambda$-connections on a curve where $\lambda$ is a formal parameter, Math. Res. Lett. 12 (2002) 551–565.

\bibitem{Bho} 
U. Bhosle, Picard group of the moduli spaces of vector bundles, Math. Ann.  314 (1999) 245-263.

\bibitem{BM} 
I. Biswas, V. Mu\~{n}oz, Torelli theorem for moduli space of $\mathrm{SL}(r,\mathbb{C})$-connections on a compact Riemann surface, Commun. Contemp. Math. 11 (2009) 1-26.

\bibitem{Car} 
S. Cardona, On vanishing theorems for Higgs bundles, Diff. Geom.  Appl.  35 (2014) 95-102.

\bibitem{ChWe} 
X. Chen, R. Wentworth, The nonabelian Hodge correspondence for balanced hermitian metrics of Hodge--Riemann type, \href{https://arxiv.org/abs/2106.09133}{arXiv:2106.09133}.

\bibitem{CW} 
B. Collier, R. Wentworth, Conformal limits and the Bialynicki-Birula stratification of the space of $\lambda$-connections,  Adv. Math. 350 (2019) 1193-1225.

\bibitem{Cor} 
K. Corlette, Flat $G$-bundles with canonical metrics, Jour. Diff. Geom. 28 (1988) 361-382.

\bibitem {Del}  
P. Deligne, Various letters to C. Simpson.

\bibitem{Don} 
S.K. Donaldson, Twisted harmonic maps and the self-duality equations, Proc. London Math. Soc. 55 (1987) 127–131.

\bibitem{DFKMMN} 
O. Dumitrescu, L. Fredrickson, G. Kydonakis, R. Mazzro, M. Mulase, A. Neitzke, From the Hitchin section to opers through nonabelian Hodge,  Jour. Diff. Geom. 117 (2021) 223-253.

\bibitem{Gai} 
D. Gaiotto, Opers and TBA, \href{https://arxiv.org/abs/2002.00358}{arXiv:1403.6137}.

\bibitem{Hit} 
N.J. Hitchin, The self-duality equations on a Riemann surface, Proc.  London Math.  Soc. 1 (1987)  59-126.

\bibitem{HKLR} 
N.J. Hitchin, A. Karlhede, U. Lindstr\"{o}m, M. Ro\v{c}ek, Hyper-K\"{a}hler metrics and supersymmetry, Comm. Math. Phys. 108 (1987) 535-589.

\bibitem{HH1} 
Z. Hu, P. Huang,  Degenerate, strong and stable Yang--Mills--Higgs pairs, Jour.  Geom. Phys. 120 (2017) 73-88.

\bibitem{HH2} 
Z. Hu, P. Huang, The Hitchin--Kobayashi correspondence for quiver bundles over generalized K\"{a}hler manifolds,  Jour.  Geom.  Anal. 30 (2020), 3641-3671.

\bibitem{HH3}
Z. Hu, P. Huang, Simpson filtration and oper stratum conjecture, manus. math. 167 (2022) 653-673.

\bibitem{Hua} 
P. Huang, Th\'eorie de Hodge non-Ab\'elienne et des sp\'ecialisations, Ph.D. Thesis, Universit\'e C\^ote d'Azur, 2020. \href{https://tel.archives-ouvertes.fr/tel-03134917}{HAL: https://tel.archives-ouvertes.fr/tel-03134917}.

\bibitem{IIS} 
M. Inaba, K. Iwasaki, M.-H. Saito, Moduli of stable parabolic connections, Riemann--Hilbert correspondence and geometry of Painlev\'{e} equation of type VI, Part I, Publ. Res. Inst. Math. Sci. 42 (2006) 987-1089.

\bibitem{LRT} 
C.-C. Liu, S. Rayan, Y. Tanaka, The Kapustin-Witten equations and nonabelian Hodge theory,  Eur. J. Math. to appear, \href{https://arxiv.org/abs/2012.06175}{arXiv:2012.06175}. Doi: \href{https://doi.org/10.1007/s40879-022-00538-4}{10.1007/s40879-022-00538-4}.

\bibitem{LT} 
M. L\"{u}bke, A. Teleman, The Kobayashi--Hitchin correspondence, World Scientific, Singapore, 1995.

\bibitem{MSWW} 
R. Mazzeo, J. Swoboda, H. Weiss, F.  Witt, Ends of the moduli space of Higgs bundles, Duke Math.  J. 165 (2016) 2227-2271.

\bibitem{TM1} 
T. Mochizuki, Kobayashi--Hitchin correspondence for tame harmonic bundles and an application,  Ast\'erisque No. 309 (2006) viii+117.

\bibitem{TM2} 
T. Mochizuki, Kobayashi--Hitchin correspondence for tame harmonic bundles, II, Geom.  Topol. 13 (2009) 359-455.

\bibitem{TM3} 
T. Mochizuki,  Asymptotic behaviour of certain families of harmonic bundles on Riemann surfaces,  J.  Topol.  9 (2016) 1021-1073.

\bibitem{TM4} 
T. Mochizuki, Good wild harmonic bundles and good filtered Higgs bundles,  SIGMA Symmetry Integrability Geom. Methods Appl. 17 (2021), 068, 66 pages.

\bibitem{CS1} 
C.T. Simpson, Constructing of variations of Hodge structure using Yang--Mills theory and applications to uniformization,  J. Amer. Math. Soc. 1 (1988) 867-918.

\bibitem{CS2} 
C.T. Simpson, Harmonic bundles on noncompact curves, J. Amer. Math. Soc. 3 (1990) 713-770.
 
\bibitem{CS3} 
C.T. Simpson, A lower bound for the size of monodromy of systems of ordinary differential equations, in: Algebraic geometry and analytic geometry, ICM-90 Satell. Conf. Proc., Springer, Tokyo, 1991, pp. 198-230.

\bibitem{CS4} 
C.T. Simpson, Higgs bundles and local systems, Inst. Hautes \'Etudes Sci. Publ. Math. 75 (1992) 5-95.

\bibitem{CS5} 
C.T. Simpson, Moduli of representations of the fundamental group of a smooth projective variety I, Inst. Hautes \'Etudes Sci. Publ. Math. 79 (1994) 47-129.

\bibitem{CS6} 
C.T. Simpson, Moduli of representations of the fundamental group of a smooth projective variety II, Inst. Hautes \'Etudes Sci. Publ. Math. 80 (1994) 5-79.

\bibitem{CS7} 
C.T. Simpson, The Hodge filtration on nonabelian cohomology, in: Algebraic geometry-Santa Cruz 1995, in: Proc. Sympos. Pure Math., vol. 62, Amer. Math. Soc., Providence, RI, 1997, pp. 217–281.  

\bibitem{CS8} 
C.T. Simpson, A weight two phenomenon for the moduli of rank one local systems on open varieties, in: From Hodge theory to integrability and TQFT $tt^*$-geometry, in: Proc. Sympos. Pure Math., vol. 78, Amer. Math. Soc., Providence, RI, 2008, pp. 175-214.

\bibitem{CS9} 
C.T. Simpson, Iterated destabilizing modifications for vector bundles with connection, in: Vector bundles and complex geometry, in: Contemp. Math., vol. 522, Amer. Math. Soc., Providence, RI, 2010, pp. 183–206.

\bibitem{UY} 
K. Uhlenbeck, S.-T. Yau, On the existence of Hermitian--Yang--Mills connections in stable vector bundles, Comm. Pure Appl. Math. 39 (1986) 257-293.


\end{thebibliography}
\end{document}